\documentclass[reqno]{amsart}
\usepackage[margin=1in]{geometry}

\usepackage[margin=1in]{geometry}

\usepackage{amsfonts}
\usepackage{amsmath} 
\usepackage{amssymb}
\usepackage{graphicx}
\usepackage[toc,page]{appendix}
\usepackage{wrapfig}
\usepackage{amsthm}
\usepackage{dsfont}    
\usepackage{microtype} 

\usepackage{caption}
\usepackage{subcaption}
\usepackage{color}
\usepackage{ esint }
\usepackage{url}
\usepackage{hyperref}

\usepackage[utf8]{inputenc}
\usepackage[english]{babel}
\theoremstyle{plain}
\newtheorem{theorem}{Theorem}
\newtheorem{lemma}{Lemma}
\newtheorem{proposition}{Proposition}
\newtheorem{corollary}{Corollary}
\newtheorem{definition}{Definition}
\newtheorem{remark}{Remark}
\newtheorem{conjecture}{Conjecture}

\newcommand{\onehalf}{\frac{1}{2}}
\newcommand{\N}{\mathbb{N}}
\newcommand{\R}{\mathbb{R}}

\newcommand{\s}{\mathbb{S}^2}

\newcommand{\Cone}{\mathcal{C}^1}

\newcommand{\Ctwo}{\mathcal{C}^2}

\newcommand{\ra}{\rightarrow}
\newcommand{\mt}{\mapsto}

\newcommand{\discrep}{\mathcal{D}}

\newcommand{\fatone}{\mathds{1}}

\newcommand{\length}{\text{length}}
\newcommand{\SO}{\mathcal{SO}(3)}

\begin{document}

\author{Damir Ferizovi{\'c}{$^{1}$}}
\thanks{$^{[1]}$ Katholieke Universiteit Leuven, Celestijnenlaan 200B, Leuven, 3001, Belgium. Email: \href{mailto:damir.ferizovic@kuleuven.be}{damir.ferizovic@kuleuven.be}.  Supported by long term structural funding -- Methusalem grant of the Flemish Government; the Austrian Science Fund (FWF): F5503 ``Quasi-Monte Carlo Methods''; and FWF: W1230 ``Doctoral School Discrete Mathematics.''}
\author{Julian Hofstadler{$^{2}$}}
\thanks{$^{[2]}$ Universität Passau, Innstraße 33, Passau, 94032, Germany. Email: \href{mailto:julian.hofstadler@uni-passau.de}{julian.hofstadler@uni-passau.de}. Supported by Deutsche Forschungsgemeinschaft (DFG) within project 432680300 - SFB 1456 (subproject B02) and the Johannes Kepler University Linz.}
\author{Michelle Mastrianni{$^{3}$}}
\thanks{$^{[3]}$ University of Minnesota, 206 Church St. SE, Minneapolis, 55455, Minnesota, USA. Email: \href{mailto:michmast@umn.edu}{michmast@umn.edu}. Supported by DMS-2054606 through the US National Science Foundation.}

\title{The spherical cap discrepancy of HEALPix points}
\begin{abstract}
In this paper we show that the spherical cap discrepancy of the point set given by centers of pixels in the HEALPix tessellation (short for Hierarchical, Equal Area and iso-Latitude Pixelation) of the unit $2$-sphere is lower and upper bounded by order square root of the number of points, and compute explicit constants. This adds to the known collection of explicitly constructed sets whose discrepancy converges with order $N^{-1/2}$, matching the asymptotic order for i.i.d. random point sets. We describe the HEALPix framework in more detail and give explicit formulas for the boundaries and pixel centers. We then introduce the notion of an $n$-convex curve and prove an upper bound on how many fundamental domains are intersected by such curves, and in particular we show that boundaries of spherical caps have this property. Lastly, we mention briefly that a jittered sampling technique works in the HEALPix framework as well.
\end{abstract}
\maketitle

\section{Introduction}
\indent HEALPix, short for ``Hierarchical, Equal Area and iso-Latitude Pixelation,'' is an algorithm introduced in \cite{Gorski} to discretize  $\s$ and has been widely used in cosmology and astrophysics. Given $\ell \in \N_0$, the algorithm produces a tessellation starting with 12 \textit{base pixels} followed by a further binary partition up to resolution $\ell$. In this paper we study the spherical cap discrepancy of the point set given by the centers of the HEALPix pixels. We discuss the tessellation algorithm in more detail in \textsection \ref{sec:pixel}  and give a precise definition of $\Gamma$, a projection of the sphere onto the plane, in \textsection \ref{subsec_Gamma}. 
\begin{definition}\label{def_HealpixPoints}
	Let $N=12\cdot 4^{\ell}$ and let $\Gamma$ be as defined in \textsection \ref{subsec_Gamma}. Then the HEALPix points $H_N$ on $\s$  are defined as the preimage under $\Gamma$ of the centers of squares  obtained by the image of the HEALPix tessellation under $\Gamma$.
\end{definition}

In order to quantify how well-distributed HEALPix points are on $\s$, we use the spherical cap discrepancy. We refer for other notions of discrepancy to the texts \cite{Matousek} and \cite{CST} and for their applications to the surveys \cite{BrauchartGrabner} and \cite{KuijlaarsSaff}. There are other ways to quantify uniform distribution. For example, one can consider the \emph{Riesz $s$-energy} of a point set: see \cite{BorodachovHardinSaff} for a comprehensive monograph on these topics.

Let $\s = \{z \in \mathbb{R}^3: \|z\| =1\}$ denote the unit sphere in the Euclidean space $\mathbb{R}^3$, where $\|\cdot\|$ is induced by the standard inner product $\langle x, y \rangle$. A \emph{spherical cap} with center $w \in\s$ and height $t\in [-1,1]$ is given by the set
$$C(w,t)=\{x \in\s\ :\ \langle x,w \rangle\geq t\}.$$


\begin{definition}[Local spherical cap discrepancy]
	Let $Z_N=\{z_1,\ldots,z_N\}\subset\s$. 
	The local spherical cap discrepancy of $Z_N$ with respect to a cap $C(w,t)$ is
	$$\mathcal{D}_{C(w,t)}(Z_N)= \Bigl\lvert\frac{1}{N}\sum_{j=1}^{N} \fatone_{C(w,t)}(z_j)-\sigma(C(w,t))\Bigr\rvert,$$
 where $\sigma$ is the normalized uniform measure on $\mathbb{S}^2$. 
 \end{definition}
 

\begin{definition}[Spherical cap discrepancy]
	Let $Z_N=\{z_1,\ldots,z_N\}\subset\s$. 
	The spherical cap discrepancy of $Z_N$ is
	$$\discrep(Z_N)=\sup_{w\in \s} \sup_{-1\leq t\leq1} \discrep_{C(w,t)}(Z_N).$$
 \end{definition}


Lubotzky, Phillips and Sarnak in \cite{LPS} were the first to achieve an upper bound for the spherical cap discrepancy of a deterministic point set, of the order $N^{-1/3} \log^{2/3} N$. Aistleitner, Brauchart, and Dick showed in \cite{BrauchartAistleitnerDick} that for i.i.d. random point sets, spherical digital nets, and spherical Fibonacci lattices, the spherical cap discrepancy is at most of order $N^{-1/2}$. The latter two point sets are deterministic, and a further such example is given by the Diamond ensemble introduced first in \cite{BE}, which has cap discrepancy of order $N^{-1/2}$ as shown in \cite{Etayo}. Beck showed in \cite{Beck1} that there exist point sets $\omega_N^{\star}$ which achieve even lower order discrepancy:
there are constants $c, C > 0$ independent of $N$ such that 
\begin{equation}
\label{eq:beckbound}
cN^{-3/4} \le 
\discrep(\omega_N^{\star}) \le CN^{-3/4}\sqrt{\log N}
\end{equation}
where the lower bound holds for all $N$-point sets. Numerical results indicate that Fibonacci lattices achieve discrepancy of an order
close to Beck's upper bound. On the other hand, probabilistic point sets with discrepancy matching this upper bound with high probability were  obtained in \cite{AlishahiZamani} and later in \cite{BMO}.

Our main result is as follows and extends Proposition 2.4 in \cite{HardinMichaelsSaff}, where $H_N$  are shown to be asymptotically uniformly distributed.

\begin{theorem}
\label{thm:main}
Let $H_N$ be the set of centers of pixels in a HEALPix tessellation of $\s$ with $N$ pixels. Then 
$$\frac{2}{\sqrt{3}}N^{-1/2} \le \discrep(H_N) \le \frac{2(5+\sqrt{2})}{\sqrt{3}} N^{-1/2} + 1000 N^{-1}.$$
\end{theorem}

To obtain the coefficient of the leading term in the upper bound, we bound the maximum number $I_{H_N}$ of pixels of the HEALPix tessellation that are intersected by spherical caps. We believe that the constant could be improved to $4/\sqrt{3}$, which would make the HEALPix points the deterministic algorithm with the lowest-known constant in the discrepancy upper bound. This would follow from the proof technique used later and from the following conjecture.

\begin{conjecture}
$I_{H_N} = 2^{3+\ell}$, where $N = 12\cdot 4^{\ell}$.
\end{conjecture}

Much of our paper is devoted to showing that $I_{H_N} \le 4(5+\sqrt{2})2^{\ell} + 1000$. We do so by projecting the boundaries of spherical caps under the HEALPix projection $\Gamma$, bounding their lengths, and employing a lemma that relates the length of a curve to the number of intersections with the projected subpixels of the HEALPix tessellation.\\

In \textsection \ref{sec:pixel} of the paper, we describe the HEALPix tessellation in more detail and give explicit formulas for the boundaries. (The appendix includes a more detailed indexing of the pixel centers for the reader's convenience.) We also parametrize the boundaries of spherical caps which are used throughout the text. \textsection \ref{sec:projection} introduces the projection $\Gamma: \s \rightarrow \mathbb{R}^2$ of HEALPix pixels into the plane and some key properties. The main theorem is proved in \textsection \ref{sec:discrepancybounds}. For the upper bound of the discrepancy, we use a bound on the length of $\Gamma(\partial C)$ for spherical caps $C$, and then employ a classical approach. The lower bound of the same order is then established by considering a particular spherical cap. Lastly, we mention that a jittered sampling technique in the HEALPix lattice yields Beck's upper bound \eqref{eq:beckbound}.  In \textsection \ref{sec:calculations}  we will introduce the notion of an $n$-convex curve and prove an upper bound on how many fundamental domains are intersected by such curves, as recently studied by the first-named author in \cite{Fer}. This upper bound depends directly on the length of the curve, so we then show a length bound for boundaries of spherical caps. The last part of the section is the final ingredient to show that the boundaries of spherical caps are $n$-convex.

\section{Boundaries of the HEALPix pixels and spherical caps}
\label{sec:pixel}
Equal area partitions of the sphere give rise to point sets with good distribution properties; examples of such partitions are given in \cite{BRV} and \cite{RSZ}. The HEALPix tessellation is another such example. The HEALPix algorithm introduces a division of the sphere into 12 spherical rectangles (\textit{base pixels}) of equal area, each of which allows a dyadic decomposition into smaller rectangles and satisfies the following three properties.

\begin{enumerate}
    \item \emph{Hierarchical structure of the database.} The resolution of a particular tessellation increases by dividing each pixel into four smaller ones. The hierarchical nature of the HEALPix algorithm allows us to easily define the center and boundaries of each pixel.
    \item \emph{Equal areas of pixels.} Ensuring equal area of pixels in a HEALPix tessellation of any resolution  is advantageous in a number of physical contexts.
    \item \emph{Iso-Latitude distribution of pixels.} As in item 1. above, this property is useful in allowing us to precisely define the centers and boundaries of the pixels.
\end{enumerate}

We introduce below the general formulas for the boundaries of pixels, following \cite{Gorski}. 

\subsection{HEALPix tessellation algorithm}
The sphere is separated into three regions: the north pole cap, the equatorial belt and the south pole cap. The lines of separation between the regions are given via intersection of $\s$ with $\mathbb{R}^2 \times \{\pm \frac{2}{3}\}$ (see Figure \ref{fig:sphereandgamma}).

We refer to the 12 base pixels, which form the basis of the subdivision algorithm, as the level 0 HEALPix tessellation $\mathcal{T}_{0}$: shown in Figure \ref{fig:tesselations} below. Note that four of the base pixels lie entirely within the equatorial belt; four pixels share a vertex at the north pole; and the southern hemisphere is a copy of the northern with negated $z$-value.

The level $\ell$ tessellation $\mathcal{T}_{\ell}$ takes a dyadic decomposition of each base pixel into $4^{\ell}$ subpixels of equal area. In total there are $12\cdot 4^{\ell}$ many pixels in $\mathcal{T}_{\ell}$. We refer to $\ell$ as the \emph{resolution parameter} or \emph{level} of a tessellation.

\begin{figure}[h!]
\centering
\includegraphics[scale=0.5]{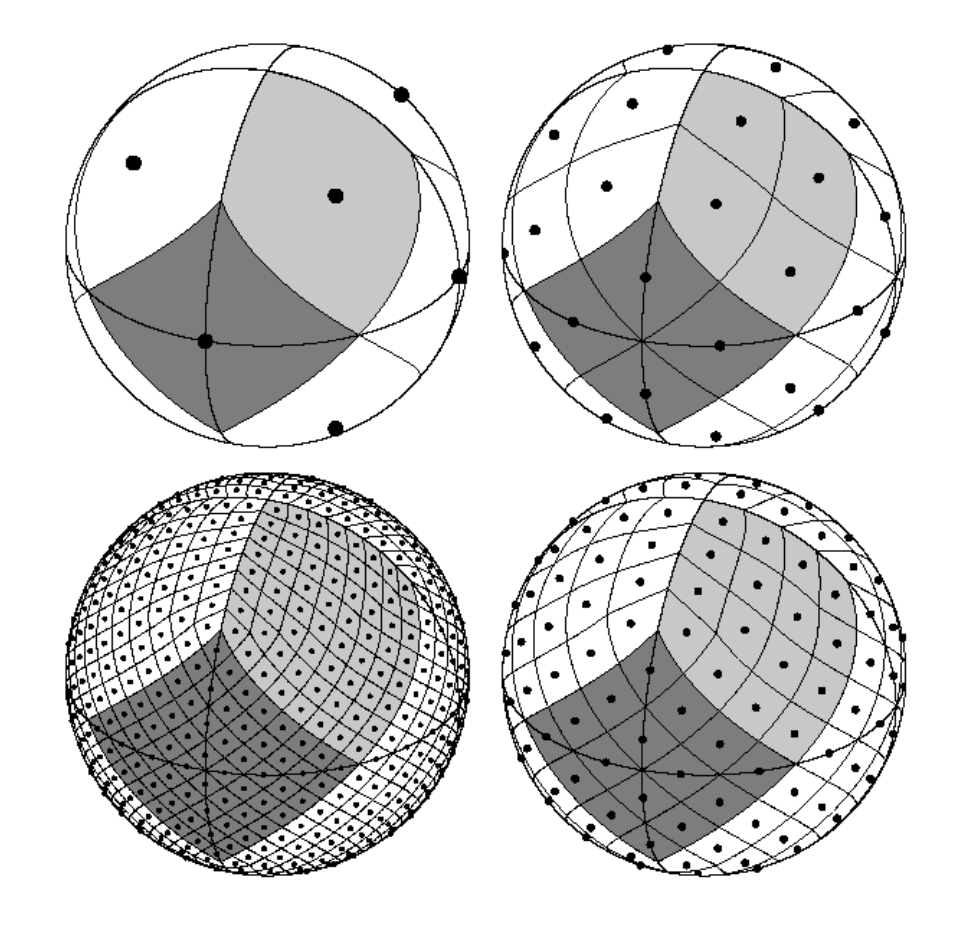} 
\caption{\centering HEALPix tessellations at levels 0, 1, 2, and 3. \copyright AAS. Reproduced with kind permission from the authors of \cite{Gorski} and The Astrophysical Journal. See DOI: 10.1086/427976.}
\label{fig:tesselations}
\end{figure}

\subsection{Boundaries of pixels in equatorial belt}
\label{subsec:boundaries}
First consider the base pixels that are entirely contained within the equatorial belt. Suppose the resolution parameter of the HEALPix tesselation is $\ell$, i.e. each base pixel is divided into $4^{\ell}$ rectangular shapes of equal area and similar diameter. Let $L = 2^{\ell}$. Then we define for $j\in\{0,1,\ldots,4L-1\}$

\begin{center}
$\begin{array}{rcl}
\Phi^{\ell}_j:I_j^{\ell} &\ra& \big[0,\frac{\pi}{2}\big] \\
\phi &\mt&\phi-\frac{j}{L}\frac{\pi}{2}
\end{array}
 \hspace{0.5cm}\mbox{with}\hspace{0.5cm} I_j^{\ell}:=\Big[\frac{j}{L}\frac{\pi}{2},\frac{j+L}{L}\frac{\pi}{2}\Big].$ 
 \end{center}
 
\noindent The pixel boundaries of resolution parameter $\ell$ are obtained via curves $m_{j,\ell}^e$ from north to south, and curves $p_{j,\ell}^e$ from south to north, implicitly given by the lines
\[m_{j,\ell}^e \sim \cos(\theta) = \frac{2}{3}-\frac{8}{3\pi}\Phi_j^{\ell}(\phi) \hspace{0.25cm} \text{  and  } \hspace{0.25cm} p_{j,\ell}^e \sim \cos(\theta) = -\frac{2}{3}+\frac{8}{3\pi}\Phi_j^{\ell}(\phi),\]
meaning that $m_{j,\ell}^e=\big(\cos( \phi)\sin(\theta),\sin(\phi)\sin(\theta),\cos(\theta)\big)$ for $\phi \in I_j^{\ell}$ and $\theta = \theta(\ell,j)$.

\subsection{Boundaries of pixels in north and south pole caps}


We will only describe the boundaries for pixels in the north pole cap, as those on the south pole cap are obtained by changing the sign of the $z$-value. The boundaries for the base pixels are given by a fixed $\phi \in \{0,\frac{\pi}{2},\pi,\frac{3\pi}{2}\}$ and all $0 \le \theta \le \arccos(\frac{2}{3})$. We index the four base pixels with vertex at the north pole by $k \in \{0,1,2,3\}$. If $\ell \ge 1$ let $L = 2^{\ell}$; we define for $k$ and $j\in\{1,\ldots,L-1\}$:
$$
\begin{array}{rcl}
\Phi^{\ell,k}_j:E_j^{\ell,k}&\ra& \big[\frac{j}{L}\frac{\pi}{2},\frac{\pi}{2} \big] \\
\phi&\mt&(k+1)\frac{\pi}{2}-\phi
\end{array}
\hspace{0.5cm}\mbox{with}\hspace{0.5cm} E_j^{\ell,k}:=\Big[k\frac{\pi}{2}, (k+1)\frac{\pi}{2}-\frac{j}{L}\frac{\pi}{2}\Big];$$
and further
$$
\begin{array}{rcl}
\psi^{\ell,k}_j:A_j^{\ell,k}&\ra& \big[\frac{j}{L}\frac{\pi}{2},\frac{\pi}{2} \big) \\
\phi&\mt&\phi-k\frac{\pi}{2}
\end{array}
\hspace{0.5cm}\mbox{with}\hspace{0.5cm} A_j^{\ell,k}:=\Big[ k\frac{\pi}{2}+\frac{j}{L}\frac{\pi}{2},(k+1)\frac{\pi}{2}\Big).$$
The pixel boundaries of resolution parameter $\ell$ are obtained via curves $m^n_{j,k,\ell}$ from north to south, and curves $p^n_{j,k,\ell}$ from south to north, implicitly given by the lines (where $\theta = \theta(j,k,\ell)$):
\begin{align*} &m^n_{j,k,\ell}\sim \cos(\theta)= 1-\frac{j^2}{3L^2}\Big(\frac{\pi}{2\Phi_j^{k,\ell}(\phi)}\Big)^2 \hspace{0.2cm} \text{ and }
\\
 & p^n_{j,k,\ell} \sim \cos(\theta)=1-\frac{j^2}{3L^2}\Big(\frac{\pi}{2\psi_j^{\ell,k}(\phi)}\Big)^2. \end{align*}
 
 \subsection{Parametrizations of spherical cap boundaries}
\label{section:curvature}

In all that follows, $\gamma: [0,2\pi)\times  (0,\pi) \rightarrow \s_p$ is the standard parametrization of the (punctured) sphere defined below: \begin{equation}
\label{eq:smallgamma}
\begin{array}{rcl}
(\phi,\theta)&\mt&\gamma(\phi,\theta)=\left(\begin{array}{l}
\cos(\phi)\sin(\theta)\\
\sin(\phi)\sin(\theta)\\
\cos(\theta)
\end{array}\right).
\end{array}
\end{equation}

(We exclude the poles to obtain an injection, but in what follows we sometimes abuse notation and plug in points of the form $(\phi, 0)$ and $(\phi, \pi)$.) Let $C = C(w,t)$ be a spherical cap with boundary $\partial C$, and $\gamma(\phi_w,\theta_w)=w\in\s$ and $t\in[0,1)$, then we have by the trigonometric angle sum formulas
\begin{equation}\label{eq_boundaryC}
	\partial C=\Big\{\gamma(\phi,\theta)\ :\ \sin(\theta)\sin(\theta_w)\cos(\phi-\phi_w)+\cos(\theta)\cos(\theta_w)=t\Big\}.
\end{equation}
Let $c_w:=\cos(\theta_w)$ and $s_w:=\sin(\theta_w)$. If $s_w=0$, then $w$ is one of the poles $p_n$ or $p_s$ and

$$ \partial C(p_{n,s})=\Big\{\gamma (\phi,\arccos\big(\tfrac{t}{c_w}\big))\ :\ \phi\in[0,2\pi)\Big\}.$$
In all other cases we can write
\begin{align}
\phi&=\phi_w+\arccos\Big(\frac{t-\cos(\theta)c_w}{\sin(\theta)s_w}\Big)\mbox{ and/or }\label{parametrization1}\\
	\phi&=\phi_w-\arccos\Big(\frac{t-\cos(\theta)c_w}{\sin(\theta)s_w}\Big)\label{parametrization2},
\end{align}
since the cosine from \eqref{eq_boundaryC} depends on $\lvert\phi-\phi_w\rvert$, and where $$\theta\in\big[\lvert \arccos(t)-\theta_w \rvert,\arccos(t)+\theta_w\big]$$ with $t\in[0,1)$ and $\theta_w\in[0,\frac{\pi}{2}]$.

As the curvature and other properties for both parametrizations behave similarly, we will be using only parametrization \eqref{parametrization1} in the curvature calculations in \textsection \ref{subsec:curvcalc} and elsewhere.  \\ 

If the north pole is in $\partial C$ but the south pole is not, then $c_w=t$, and since 
$1-\cos(x)=\tan\big(\frac{x}{2}\big)\sin(x)$, we obtain from \eqref{parametrization1}
\begin{equation}\label{eq_northpoleparametrization}
	\phi=\arccos\Big(\frac{ t}{\sqrt{1-t^2}}\tan\big(\tfrac{\theta}{2}\big)\Big)+\phi_w.
\end{equation}
 If the south pole is in $\partial C$ but the north pole is not, then $c_w = -t$, and since $1+\cos(x)=\cot\big(\frac{x}{2}\big)\sin(x)$  we obtain
 \begin{equation}\label{eq_southpoleparametrization}
 	\phi=\arccos\Big(\frac{ t}{\sqrt{1-t^2}}\cot\big(\tfrac{\theta}{2}\big)\Big)+\phi_w.
 \end{equation}
 If the north and south poles are in $\partial C$, then clearly $\theta_w=\frac{\pi}{2}$ and \eqref{eq_boundaryC} is equivalent to 
 $$\sin(\theta)\cos(\phi-\phi_w)=t=0,$$
  and the only possibility for a circle satisfying this relation and going through the poles is to have  $$\partial C = \gamma\Big(\{\phi_w+\frac{\pi}{2},\phi_w+\frac{3\pi}{2}\}\times[0,\pi]\Big).$$
   

\section{Projection into the Plane}

\label{sec:projection}

In this section we define the projection map $\Gamma$ of the HEALPix tessellation into the plane, as defined in \cite{Gorski}, show that it preserves area up to scaling, and study boundaries of spherical caps under $\Gamma$. 

Each of the eight base pixels which have one vertex at a pole are simply rotated copies of the others. The remaining four base pixels have centers on the equatorial belt, and again each of these four base pixels are rotated copies of the others.  Thus we have some rotational symmetry. 

 \subsection{Definition of $\Gamma$}\label{subsec_Gamma}
 
 The projection map
 $$\begin{array}{rcl}
\Gamma:\s_p &\ra& R \\
x&\mt&\Gamma(x)
\end{array}
$$
 from the sphere to $R:=[0,2)\times (-\onehalf,\onehalf)$ 
 is defined in the following fashion, where $\gamma$ is as in \eqref{eq:smallgamma}.
 
\begin{itemize}
\item If $\lvert \cos(\theta)\rvert \leq\frac{2}{3}$, then  	
$$\begin{array}{rcl} x=\gamma(\phi,\theta)&\mt&\Gamma(x)=\left(\begin{array}{l} \phi/\pi \\ 3/8\cos(\theta) \end{array}\right).
	\end{array}
	$$
\item If $\cos(\theta)>\frac{2}{3}$, then
	$$
	\begin{array}{rcl}
	x=\gamma(\phi,\theta)&\mt&\Gamma(x)=\dfrac{1}{\pi}\left(\begin{array}{l}
	\phi-\big(1-\sqrt{1-\cos(\theta)}\sqrt{3}\big)\cdot(\phi\hspace{-0.2cm}\mod \frac{\pi}{2}-\frac{\pi}{4})\\
	\frac{\pi}{4}\big(2-\sqrt{1-\cos(\theta)}\sqrt{3}\big)
	\end{array}\right).
	\end{array}
	$$
\item If $\cos(\theta)<-\frac{2}{3}$, then
	$$
	\begin{array}{rcl}
	x=\gamma(\phi,\theta)&\mt&\Gamma(x)=\dfrac{1}{\pi}\left(\begin{array}{l}
	\phi-\big(1-\sqrt{1-\lvert\cos(\theta)\rvert}\sqrt{3}\big)\cdot(\phi\hspace{-0.2cm}\mod \frac{\pi}{2}-\frac{\pi}{4})\\
	\frac{\pi}{4}\big(\sqrt{1-\lvert\cos(\theta)\rvert}\sqrt{3}-2\big)
	\end{array}\right).
	\end{array}
	$$
\end{itemize}

\begin{figure}[h!]
\begin{center}
\includegraphics[scale=0.7]{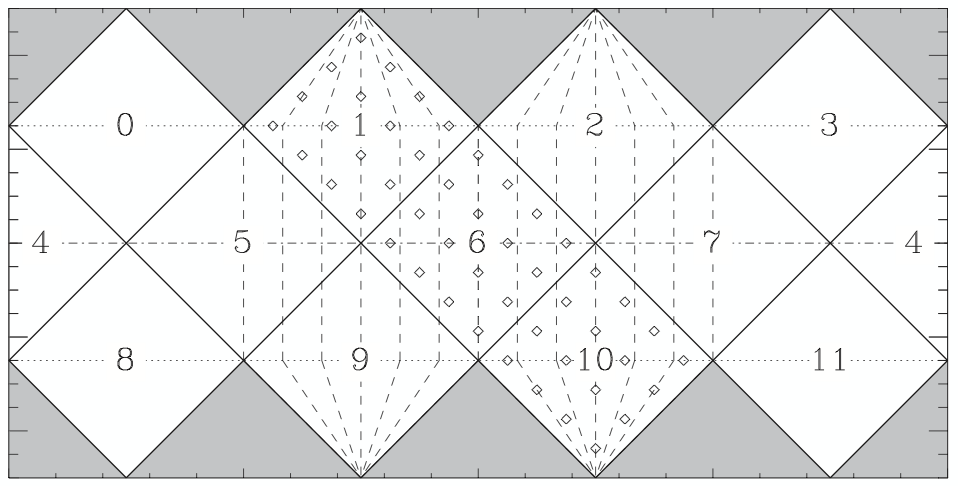}
\caption{\centering Projection of HEALPix lattice. \copyright AAS. Reproduced with kind permission from the authors of \cite{Gorski} and The Astrophysical Journal. See DOI: 10.1086/427976.}
\label{fig:projection}
\end{center}
\end{figure}

\begin{remark}[See \cite{Gorski}]
		Pixel boundaries are mapped under $\Gamma$ to straight line segments.
\end{remark}
Each pixel is mapped to a square in $R$, as the segments $\Gamma(m^e_{j,\ell})$  are parallel to each other and seperated by the same distance. The same is true for $\Gamma(p^e_{j,\ell})$, which are orthogonal to $\Gamma(m^e_{j,\ell})$. The same applies to  $\Gamma(m^n_{j,k,\ell})$ and $\Gamma(p^n_{j,k,\ell})$. The resulting lattice (see Fig. \ref{fig:projection}) of pixel centers is the $\mathbb{Z}^2$ lattice scaled and rotated.


We use the following fact, which follows from the definition of $\Gamma$, implicitly in further calculations.
\begin{remark}
		Boundaries of hemispheres going through the poles are mapped to simple polygons by $\Gamma$. For most other caps $C$, if $\Gamma(\partial C)$ intersects such polygons, it does so at most twice, hence restricting the geometry of $\Gamma(\partial C)$. 
\end{remark}

\subsection{Preservation of area up to scaling}
 It is shown in \cite{Gorski} that pixels of the same resolution $\ell$ have the same area. Let $B_1, \cdots, B_{12}$ denote the 12 base pixels on the sphere. Then $\Gamma(B_1), \Gamma(B_2), \cdots, \Gamma(B_{12})$ are squares in $R$ (if we think of $R$ as a cylinder), and they are all of equal area. The tiling in a base pixel with resolution parameter $\ell$ is mapped to a tiling in the associated square of the same resolution, and each of the subpixels at resolution $\ell$ has equal area. Any open subset $A$ of a base pixel $B_i$ can now be approximated from inside and outside via a cover of subpixels of ever increasing resolution; thus we obtain the following lemma.
\begin{lemma}
	Given a base pixel $B$ and an open set $A\subset B$, then
	$$ \frac{\mathrm{area}(A)}{\mathrm{area}(B)}=\frac{\mathrm{area}(\Gamma(A))}{\mathrm{area}(\Gamma(B))}.$$
\end{lemma}


\section{Discrepancy bounds}

\label{sec:discrepancybounds}

In this section we prove the main theorem, using ideas about intersections of curves with lattice boxes that go back to Gauss (e.g. the Gauss circle problem). We leave the proofs of some necessarily lemmas to \textsection \ref{sec:calculations}.

\subsection{Proof of upper bound of Theorem \ref{thm:main}}
\label{discrepancybound}

The next proposition, whose proof is in \textsection \ref{sec:calculations}, bounds the length of projected spherical cap boundaries under $\Gamma$.

\begin{proposition}\label{prop:length_bound_cap}
	Let $w\in\s, t\in(-1,1)$ be chosen and let $\Gamma: \s \rightarrow R$ be as in \textsection \ref{sec:projection}. Set $\beta=\Gamma(\partial C(w,t))$, then 
	$$
	\mbox{length}(\beta)
	\leq 5+\sqrt{2}
	< 6.42.
	$$
\end{proposition}

The length bound translates directly to the error coefficient in the discrepancy, as we see below. Let $H_N$ be the HEALPix points, i.e., the centers for some tessellation with a total amount of $N$ pixels. 
\begin{lemma}\label{lemma:amount_intersections}
For any spherical cap $C=C(w, t)$ the boundary of $C$ intersects at most $\frac{2(5+\sqrt{2})}{\sqrt{3}} \sqrt{N} + 1000$ (sub)-pixels.
\end{lemma}

The proof of Lemma \ref{lemma:amount_intersections} is also left to \textsection \ref{sec:calculations}. We now show the upper bound of Theorem \ref{thm:main}:
\[
D(H_N) \leq \frac{2(5+\sqrt{2})}{\sqrt{3}} \frac{1}{\sqrt{N}} + \frac{1000}{N}
.\]

\begin{proof}
Let $C$ be an arbitrary spherical cap. 
We fix some ordering of the HEALPix points, $H_N=\{p_1, \dots, p_N\}$.
To each (mid) point $p_i$ there corresponds a pixel which we denote by $P_i$. 
For neighbouring pixels we assign the boundary between them to exactly one pixel. 
In this way we have that pixels are disjoint, while still satisfying $\sigma(P_i)=\sigma(P_j)$ for any $1 \leq i,j \leq N$.

By disjointness $\sigma(C) = \sum_{i=1}^N\sigma(C \cap P_i)$ and consequently
\begin{align*}
\left\vert \frac{1}{N}\sum_{i=1}^{N} \mathds{1}_{C}(p_i) - \sigma(C) \right\vert
&=
\left\vert\sum_{i=1}^{N}\left(  \frac{1}{N}\mathds{1}_{C}(p_i) - \sigma(C\cap P_i)\right)  \right\vert \\  
& \leq
 \sum_{i=1}^{N}\left\vert \frac{1}{N}  \mathds{1}_{C}(p_i) - \sigma(C\cap P_i)  \right\vert.
\end{align*}
If either $C \cap P_i = P_i$, or $C \cap P_i = \varnothing$, then $\vert \frac{1}{N} \mathds{1}_{C}(p_i) - \sigma(C\cap P_i) \vert =0$, since HEALPix provides an equal area partition.
Otherwise, i.e., when $\partial C$ intersects $P_i$, one has $\vert \frac{1}{N}\mathds{1}_{C}(p_i) - \sigma(C\cap P_i) \vert \leq \frac{1}{N}$. 
So by Lemma \ref{lemma:amount_intersections} we can bound the local discrepancy as follows
\begin{align*}
\sum_{i=1}^{N}\left\vert  \frac{1}{N} \mathds{1}_{C}(p_i) - \sigma(C\cap P_i)  \right\vert
\leq
\frac{2(5+\sqrt{2})}{\sqrt{3}} \frac{1}{\sqrt{N}} + \frac{1000}{N}. 
\end{align*}
Since this estimate holds for arbitrary spherical caps the result follows.  
\end{proof}

\subsection{Proof of lower bound of Theorem \ref{thm:main}}

The lower bound for $\discrep(H_N)$ is easily achieved by considering the specific spherical cap $C$ that is centered at the north pole $p_n$ and has height $t =- \frac{1}{N}$: that is, it covers slightly more than half of the surface area of the sphere. 
$\lvert H_N \rvert = N=12\cdot 4^{\ell}$, so in each of the 4 base pixels on the equatorial belt, there are $4^{\ell}$ subpixels. Hence there are $2^{\ell}$ subpixels on the ``diagonal'' of each base pixel, and their centers are on the equator. Multiplying by 4, we obtain that the equator contains $\sqrt{\frac{4N}{3}}$ points in $H_N$. Thus, the local cap discrepancy satisfies
$$\discrep_{C\big(p_n,-\tfrac{1}{N}\big)}(H_N) =\frac{2}{\sqrt{3}}N^{-1/2}+O(N^{-1}).$$


\subsection{Remark on jittered sampling in HEALPix pixels} We remark that if the point set $H_N$ of centers of HEALPix pixels is replaced by a point set $J_N$ obtained by choosing a point at random from each pixel, as in jittered sampling methods, we obtain the following upper bound matching Beck's:
\[\discrep(J_N) \le CN^{-3/4}\sqrt{\log N}.\]
This can be seen roughly as follows. Replace the class $\mathcal{C}$ of all spherical caps $C(w,t)$ with those who have rational $w$ and $t$: call this class $\mathcal{C}_{\mathbb{Q}}$. The discrepancy with respect to $\mathcal{C}_{\mathbb{Q}}$ is the same as the discrepancy with respect to $\mathcal{C}$, but has finite VC dimension (see \cite[Proposition 8]{BrauchartAistleitnerDick}). 

We can now apply the jittered sampling upper bound for discs given in \cite[Theorem 3.1]{Matousek} --- which was originally shown by Beck in \cite{BC87} --- to the class $\mathcal{C}_{\mathbb{Q}}$ instead. All steps of the proof translate. As we have argued in Lemma \ref{lemma:amount_intersections}, for a spherical cap $C(w,t) \in \mathcal{C}_{\mathbb{Q}}$, the curve $\Gamma(\partial C(w,t))$ intersects at most $O(\sqrt{N})$ pixels; this gives the analogue of \cite[Lemma 3.2]{Matousek}. The analogue of \cite[Lemma 3.3]{Matousek} is obtained from the applying the Sauer-Shelah lemma (see \cite{Sauer}, \cite{Shelah}) to $\mathcal{C}_{\mathbb{Q}}$. As this class has finite VC-dimension, we only need to consider polynomially many elements (in $N$) of $\mathcal{C}_{\mathbb{Q}}$. Thus we obtain discrepancy on the order of $N^{-3/4}\sqrt{\log N}$ with respect to the class $\mathcal{C}_{\mathbb{Q}}$, and hence also with respect to $\mathcal{C}$.

\section{Extended calculations}

\label{sec:calculations}

In \textsection \ref{subsec:convexcurve} we will introduce the notion of an $n$-convex curve and prove an upper bound on how many fundamental domains are intersected by such curves, and prove Lemma \ref{lemma:amount_intersections} up to the constant, which is given by Proposition \ref{prop:length_bound_cap}. We also show in Corollary \ref{cor_capsconvex} that it is enough to regard zeros of curvature to show that the boundaries of spherical caps are $n$-convex. In \textsection \ref{subsec:prop1proof} we prove Proposition \ref{prop:length_bound_cap} by case distinction. This will then give the boundary contribution in the proof of Theorem \ref{thm:main}. Finally, in \textsection\ref{subsec:curvcalc} we show that the boundaries of spherical caps are $n$-convex by studying their curvature.

\subsection{Intersection of convex curves with fundamental domains}
\label{subsec:convexcurve}

Most of the contents of this section also appear verbatim in a work of the first-named author \cite{Fer}. Its original purpose was to prove Lemma \ref{lemma:amount_intersections}, but in a different work he extended and streamlined the result. It is reproduced here for reviewing purposes.

	\begin{definition}\label{def_nconvex}
	A continuous curve $\beta:[a,b]\ra \mathbb{R}^2$ ($a,b\in\mathbb{R}$, $a<b$) is $n$-convex for $n\in\mathbb{N}$, if there are points $a=t_0<t_1<\ldots<t_{n-1}<t_{n}=b$ with the property that  $\beta\big([t_j,t_{j+1}]\big)$ is a convex curve  for each $0\leq j< n$, i.e. there exist convex sets $A_1,\ldots,A_n$ with $\beta\big([t_j,t_{j+1}]\big)\subset \partial A_j$.

\end{definition}
A circle is $1$-convex, as is a straight line segment. 
We will give a characterization of $n$-convex regular $\Ctwo$-curves in Lemma \ref{lem_constructingNconvexPartition}   in terms of sub-spirals, which are defined next.
\begin{definition}
	A sub-spiral $S$ is a $\Ctwo$-curve with compact domain, no self-intersections, and such that the curvature exists, is bounded, and does not change sign.
\end{definition}
A spiral is a sub-spiral, but a circle is not, as it intersects itself at one point. Recall that a curve $\beta$ is called regular if $\|\dot{\beta}(t)\|>0$ everywhere on its domain, where the derivative at boundary points is defined via one-sided limits.
\begin{lemma}\label{lem_constructingNconvexPartition}
	A  regular $\Ctwo$-curve $\beta$ is $n$-convex if and only if it  consists of finitely many sub-spirals. 
\end{lemma}
\begin{proof} 	If $\beta$ is $n$-convex, take arbitrary $c_0,\ldots,c_{n}$ with $t_{j}<c_{j}<t_{j+1}$ (the $t_j$ as in Definition \ref{def_nconvex}), then clearly each $\beta\big([t_j,c_{j}]\big)$ and $\beta\big([c_j,t_{j+1}]\big)$ is a sub-spiral.
	
To show the other implication, let us first assume $\beta:[a,b]\ra \mathbb{R}^2$ is a sub-spiral itself and we define for each $\tau\in[a,b]$
	$$\dot{\beta}^-(\tau)=\lim_{\epsilon\ra0^+}\frac{\beta(\tau)-\beta(\tau+\epsilon)}{\epsilon}.$$
	 By regularity we have $\|\dot{\beta}^-(\tau)\|>0$, and we define the half-line
	$$L(\tau)=\big\{T(\rho,\beta(\tau))\  \mid \ \rho\geq 0 \big\}\hspace{0.3cm}\mbox{where}\hspace{0.3cm}T(\rho,\beta(\tau))=\beta(\tau)+\rho\dot{\beta}^-(\tau).$$
	We will now introduce an algorithm to find all necessary $t_j$'s:
	
	\vspace{0.2cm}
	
	\textbf{Loop 1.} 	If $L(a)$ doesn't intersect $\beta$, proceed to Loop 2. Otherwise let 
	$$\rho_1=\min_\rho\{T(\rho,\beta(a))\cap\beta\big([a,b]\big)\neq\emptyset\}.$$
	If $T(\rho_1,\beta(a))=\beta(b)$, proceed to Loop 2. Otherwise define $t_1$ so that $\beta(t_1)=T(\rho_1,\beta(a))$. If 
	$L(t_1)$ doesn't intersect $\beta$, proceed to Loop 2. Else set  		$$\rho_2=\min_\rho\{T(\rho,\beta(t_1))\cap\beta\big([a,b]\big)\neq\emptyset\},$$
	and $t_2$ so that $\beta(t_2)=T(\rho_2,\beta(t_1))$. By construction, 
	 $\beta([t_1,t_2])$ does not intersect the line segment $T\big([0,\rho_1],\beta(a)\big)$, and thus it coils around $\beta(a)$.
	If $\beta(t_2)=\beta(b)$, proceed to Loop 2. Otherwise regard $L(t_2)$ and define $\rho_3,t_3$ similarly or proceed to Loop 2. It then follows from constant sign of curvature, no self-intersection and choices of $\rho_s$ that $\beta([t_2,t_3])$ coils around $\beta(a)$ and $\beta(t_1)$. Continuing in this fashion we see that this process has to terminate since $\beta$ has finite length, and we proceed to Loop 2.
	
	\textbf{Loop 2.} Let $\{t_1,\ldots,t_k\}$ be the values generated in Loop 1 (potentially empty). We now reverse the parametrization: $\beta^{\leftarrow}:[a,b]\ra\mathbb{R}^2$ with $\beta^{\leftarrow}(\tau)=\beta(a+b-\tau)$, and proceed as in Loop 1  for the curve $\beta^{\leftarrow}$, where we replace the command "proceed to Loop 2" by ``terminate algorithm". 
		
	
	This way potentially further points are generated, say $\{t_{k+1},\ldots,t_{k+m}\}$, where $t_{k+m}=b$, and we see that $\beta$ is $(k+m)$-convex.
	If $\beta$ is the concatenation of sub-spirals $S_k$, then we apply the algorithm above to each $S_k$, where we collect all the  values $t_j$ in one list and  extend it by  the points of separation between $S_k$ and $S_{k+1}$ if necessary.
\end{proof}
To see if a $\Ctwo$-curve is $n$-convex it is hence enough to find the points of self-intersection and zeros of curvature where the sign changes -- the arcs between these points then are sub-spirals, and the algorithm above will find $n$.

\begin{corollary}\label{cor_capsconvex}
	The images $\Gamma\big(\partial C(w,t)\big)$ for all $(w,t)\in\s\times(-1,1)$ consist of finitely many pieces which are at most $36$-convex (and the number of pieces is bounded for all $C(w,t)$).
\end{corollary}
\begin{proof}
	This follows from Lemma \ref{lem_constructingNconvexPartition}, the definition of $\Gamma$, and the curvature calculations in \textsection \ref{subsec:curvcalc}. (Subspirals of $\Gamma(\partial C)$ are 1-convex curves themselves).
	Note that spherical caps with centers at the poles and spherical caps with centers on the equatorial belt and $t = 0$ are mapped to polygons under $\Gamma$, and in this case there are finitely many distinct curves which are $3$-convex.
\end{proof}

\begin{definition}
	Let $Q\mathbb{Z}^2$ be a lattice, where $Q$ is an invertible matrix; $\beta:[a,b]\rightarrow\mathbb{R}^2$ a continuous curve; and $K\in\mathbb{N}$. We further define $\Omega^Q_v=Q[0,1)^2+v$ for  some arbitrary but fixed $v\in\mathbb{R}^2$. The intersection number of $\beta$ with the tiling induced by  $\frac{1}{K}\big(Q\mathbb{Z}^2+ \Omega^Q_v\big)$ is 
	$$I_{\beta}^Q(K)=\#\big\{p\in Q\mathbb{Z}^2 \mid \big(\tfrac{1}{K}\Omega^Q_v+\tfrac{1}{K}p\big)\cap\beta\big([a,b]\big)\neq\emptyset \big\}.$$

\end{definition}
Clearly $I_\beta^Q(K)$ depends on the choice of $v$, but we will give a bound of it which does not.
In the next lemma we assume $\beta$ to have finitely many self-intersections to avoid parametrizations that ``walk" the same path a multitude of times.

\begin{lemma}\label{lem_sharplengthbound}
	Let $\beta$ be a piece-wise $\Cone$-curve in $\R^2$,  $n$-convex with $m$-many self-intersections for $n,m\in\N_0$. Let $\Lambda^Q$ be a full rank  lattice, then 
	$$I_\beta^{Q}(K)\leq \sqrt{2}\cdot K\cdot \mathrm{length}\big(Q^{-1}\beta\big)  +19n-m+1. $$
\end{lemma}
\begin{proof} 
	We can assume $v=0$ in the previous definition, by taking $\gamma=\beta -v$ if necessary. Thus without loss of generality, let $\beta:[0,1]\ra\R^2$ be a parametrization with finitely many self-intersections and $v=0$.   We can cut $\beta$ into sub-curves if necessary, so we further  assume $\beta$ to be $\Cone$. 	The $x,y$-coordinates of $\beta$ are denoted by $x(t)=\langle \beta(t),e_1\rangle$ and $y(t)=\langle \beta(t),e_2\rangle$. 
	
	\vspace{0.3cm}
	
	We will first show monotonicity of the coordinates for $\beta$ in certain intervals. 
	Let  $t_1,\ldots,t_n$ be as in Definition \ref{def_nconvex}, and $M_j,m_j\in \beta\big([0,t_1]\big)$  satisfy
	$$\langle M_j,e_j\rangle=\max_{t\in [0,t_1]}\langle \beta(t),e_j\rangle
	\hspace{0.2cm} \mbox{ and }\hspace{0.2cm}
	\langle m_j,e_j\rangle=\min_{t\in [0,t_1]}\langle \beta(t),e_j\rangle.$$ 
	Choose  $o_j^M\in\beta^{-1}(M_j)\cap[0,t_1]$ and $o_j^m\in\beta^{-1}(m_j)\cap[0,t_1]$ for $j\in\{1,2\}$. Let $o_1,o_2,o_3,o_4$ denote these $o_j^m,o_j^M$ in ascending order. In the calculations below the pair $(\tau_1,\tau_2)$ is $(o_k,o_{k+1})$ for some $k\in\{0,1,2,3,4\}$ with $\tau_2-\tau_1>0$, where $o_0=0$ and $o_5=t_1$.

	It follows that the $x,y$-coordinates of $\beta(t)$ are monotone for $t\in[\tau_1,\tau_2]$  by the convexity assumption:  for any $\kappa\in[0,t_1]$ and $\tau\in\{s\in[0,t_1]\ |\ \beta(s)\neq\beta(\kappa)\}$ we either have that the arc segment of $\beta$ with domain between $\tau$ and $\kappa$ is a straight line, or 
	$$\big\{ (1-s)\beta(\kappa)+s\beta(\tau)\ | \ 0<s<1 \big\}\cap \beta\big([0,t_1]\big)=\emptyset.$$
	 To prove monotonicity,  we work out the example  $\beta(\tau_1)=m_2$.
	 The existence of numbers $\tau_1<\epsilon_1<\epsilon_2<\tau_2$ such that $y(\epsilon_1)>y(\epsilon_2)$ will lead to a contradiction in this case (by 
	 the assumption on  $\beta(\tau_1)$, $y(t)$ is monotonously increasing): either  $x(\epsilon_1)\leq x(\epsilon_2)$ or $x(\epsilon_2)< x(\epsilon_1)$, in both cases the polygon with vertices and edges of the form  
	 $$\beta(o^M_2)\longrightarrow\beta(\tau_1)\longrightarrow\beta(\epsilon_1)\longrightarrow\beta(\epsilon_2) $$
	 contradicts convexity (we have $y(o^M_2)\geq y(\epsilon_1)>y(\epsilon_2)\geq y(\tau_1)=y(o^m_2)$).
	 Thus  $y(t)$ is monotonously increasing, and we use this fact to show that $x$ is monotone. Assume there are numbers $\tau_1<\epsilon_1<\epsilon_2<\epsilon_3<\tau_2$ so that $x(\epsilon_2)< x(\epsilon_1)$ and $x(\epsilon_2)< x(\epsilon_3)$, then  the following  polygon with vertices and edges of the form (depending on the orientation, we take the symbol "$*$" to be either $m$ or $M$):
	 \begin{align*}
	 	\beta(o^*_1)&\longrightarrow\beta(\epsilon_1)\longrightarrow\beta(\epsilon_2)\longrightarrow\beta(\epsilon_3),
	 \end{align*}
	 contradicts convexity.
	 The case: $x(\epsilon_1)< x(\epsilon_2)$ and $x(\epsilon_3)< x(\epsilon_2)$ is similar.
	 
	  
	  The other choices for $\beta(\tau_j)$ are analogous ($\tau_j=o_k$ for some $k\in\{1,2,3,4\}$), only the order in which one shows monotonicity of $x(t),y(t)$ depends on  which coordinate is extremized  by $\beta(\tau_j)$. 

\vspace{0.2cm}
	
	 We define two supporting axes  for each non-trivial arc $\beta([\tau_1,\tau_2])$ as follows: 	 
		$$
		\big\{ (1-t)\beta(\tau_j)+tZ \ |\ 0\leq t\leq 1 \big\}
	\hspace{0.3cm}\mbox{where}\hspace{0.3cm}
	Z=\binom{\langle \beta(\tau_1),e_1\rangle}{\langle \beta(\tau_2),e_2\rangle}\in\R^2.
	$$
	These line segments above are parallel to the axes, and we denote them accordingly by $L_x$ and $L_y$. Let $\gamma=\beta([\tau_1,\tau_2])$ and $Q$ be the identity matrix, then
	$$ I_\gamma^{\fatone}(K)\leq I_{L_x}^{\fatone}(K)+I_{L_y}^{\fatone}(K)$$
	by following argument: the coordinates of $\beta(t)$ are monotonous  for $t\in[\tau_1,\tau_2]$ (say both are increasing), then $\beta$ exits a fundamental domain $\frac{1}{K}\Omega^\fatone_0+\frac{1}{K}w$, where $w\in\Lambda^\fatone$, only by leaving it trough the top or right side.
	\begin{enumerate}
		\item 	If $\beta$ leaves through the right side, $I_\gamma^{\fatone}(K)$ and $I_{L_x}^{\fatone}(K)$ increase by one, 
		\item or else $I_\gamma^{\fatone}(K)$ and $I_{L_y}^{\fatone}(K)$ increase by one.
	\end{enumerate}
	 Thus, with $\mathrm{length}(L_x)=c$ and $\mathrm{length}(L_y)=d$, we have
	$$ I_\gamma^{\fatone}(K)\ \leq\  I_{L_x}^{\fatone}(K)+I_{L_y}^{\fatone}(K)\ \leq\  (c+d)K+4.$$
	We further use the inequality $c+d\leq \sqrt{2}\sqrt{c^2+d^2}$ valid for all $c,d\in\R$, to derive
	$$I_\gamma^{\fatone}(K)\ \leq\ \sqrt{2}\sqrt{c^2+d^2}\cdot K+4\ \leq\ \sqrt{2}\ \mathrm{length}(\gamma)\cdot K+4,$$
	where we used that the shortest path between $\beta(\tau_1)$ and $\beta(\tau_2)$ has length 	$\sqrt{c^2+d^2}$.
	The same reasoning applies to all sub-intervals (which are at most five), and after summing up and taking into account that every self-intersection counts a certain domain twice, we obtain 
	$$I_\beta^{\fatone}(K)\leq \sqrt{2}\  \mathrm{length}(\beta)\cdot K+19n+1-m.$$
	We will reduce the general case to the one above, by regarding the curve $\gamma=Q^{-1}\beta$. It is clear that $I_\gamma^{\fatone}(K)= I_\beta^{Q}(K)$.  Regularity of a curve is not affected by invertible matrices, neither is the notion of $n$-convexity nor the values $t_1,\ldots,t_n$.
\end{proof}
Note that the constant $\sqrt{2}$ in Lemma \ref{lem_sharplengthbound} cannot be improved, as the example of the translated diagonal of $I^2$ with $Q=\fatone$ already shows.

\begin{proof}[Proof of Lemma \ref{lemma:amount_intersections}]
Note that after applying the projection $\Gamma$ any sub-pixel is a square in $\mathbb{R}^2$ with diagonals parallel to the axes. 
In order to apply Lemma \ref{lem_sharplengthbound} we rescale the rectangle $R$ by $1/ \sqrt{2}$ in each coordinate, so that any base pixel has side-length $\frac{1}{4}$.
Furthermore, we rotate the scaled rectangle to obtain that any pixel is axis-parallel. 

Each base pixel contains exactly $N/12$ sub pixels and by construction we have $N/12=2^{\ell}$ for some $\ell \in \mathbb{N}$.
Boundaries of sub pixels are equally spaced lines inside each base pixel. 
This implies that (after the rotation and rescaling) corners of sub pixels are contained in the lattice $2^{-\ell-2}\mathbb{Z}^2$.

After the rescaling, the length of $\Gamma(\partial C)$ changes by a factor of $1/\sqrt{2}$. By Corollary \ref{cor_capsconvex} and Lemma \ref{lem_sharplengthbound} with $K= 2^{\ell+2}$, we obtain the result.
\end{proof}

	\subsection{Proof of Proposition \ref{prop:length_bound_cap}}
	
\label{subsec:prop1proof}	

We begin with the following lemma, which is crucial to proving Proposition \ref{prop:length_bound_cap}. The reader should refer to Figure \ref{fig:trapezoid} for the setting of the lemma.

\begin{lemma}\label{lem_trapezoidbound}
	Suppose $\beta$ is a rectifiable curve with endpoints at opposite corners of a trapezoid; the intersection of each ray $r_j$ from $f$ with $\beta$ is a connected set;  $a = \text{\length}(A)$ and $c = \text{length}(C) > d = \text{length}(D)$;  and $\lvert \text{slope}(A) \rvert \le \lvert\text{slope}(B) \rvert$, where $r_j, f, \beta, A, B, C, D$ are as in Figure \ref{fig:trapezoid}. Assume further that $\beta$ is monotone in the $y$-coordinate. Then
	$$\mbox{length}(\beta)\leq c+a.$$
\end{lemma}
\begin{proof}

	 If $\beta$ is a straight line, we use the well known formula
	$$\mbox{length}(\beta)=\sqrt{c^2+b^2-2cb\cos(\angle)}\leq c+b\leq c+a.$$
	Next, for general $\beta$ and $m\in\N$, we distribute $m+1$ equi-distant points $x_j$ (with $x_0,x_m$ being corners) on the line-segment $C$, and connect each $x_j$ with $f$ by a straight line. This way we obtain $m$-many sub-trapezoids with corners in intersection points of $\beta$ with the rays $r_j$. Let the sides of the $j$th such trapezoid be denoted by $A_j,B_j,C_j,D_j$ as depicted in the right-hand side of Figure \ref{fig:trapezoid}, and their lengths by $a_j, b_j, c_j, d_j$ respectively. Note that the slope $A_j$ is bigger than that of $A$, and the bottom sides $C_j$ are of length smaller than $c/m$. Thus the length of the diagonal $\rho_j$ of this sub-trapezoid satisfies
	$$ \mbox{length}(\rho_j)\leq c_j+\max\{a_j,b_j\}\leq \frac{c}{m}+ a_j',$$
	where $a_j'$ is the length of the projection of $A_j$ onto $A$.
	We sum both sides over $j$, and as $m$ increases to infinity, the right hand side always sums to $c+a$, but the left hand side will converge to the length of $\beta$.
\end{proof}

We remark that Lemma \ref{lem_trapezoidbound} extends to the case where $d > c$, the roles of $A$ and $B$ are interchanged, by mirror symmetry.

\begin{figure}[h]
	\centering
	\includegraphics[width=0.8\linewidth]{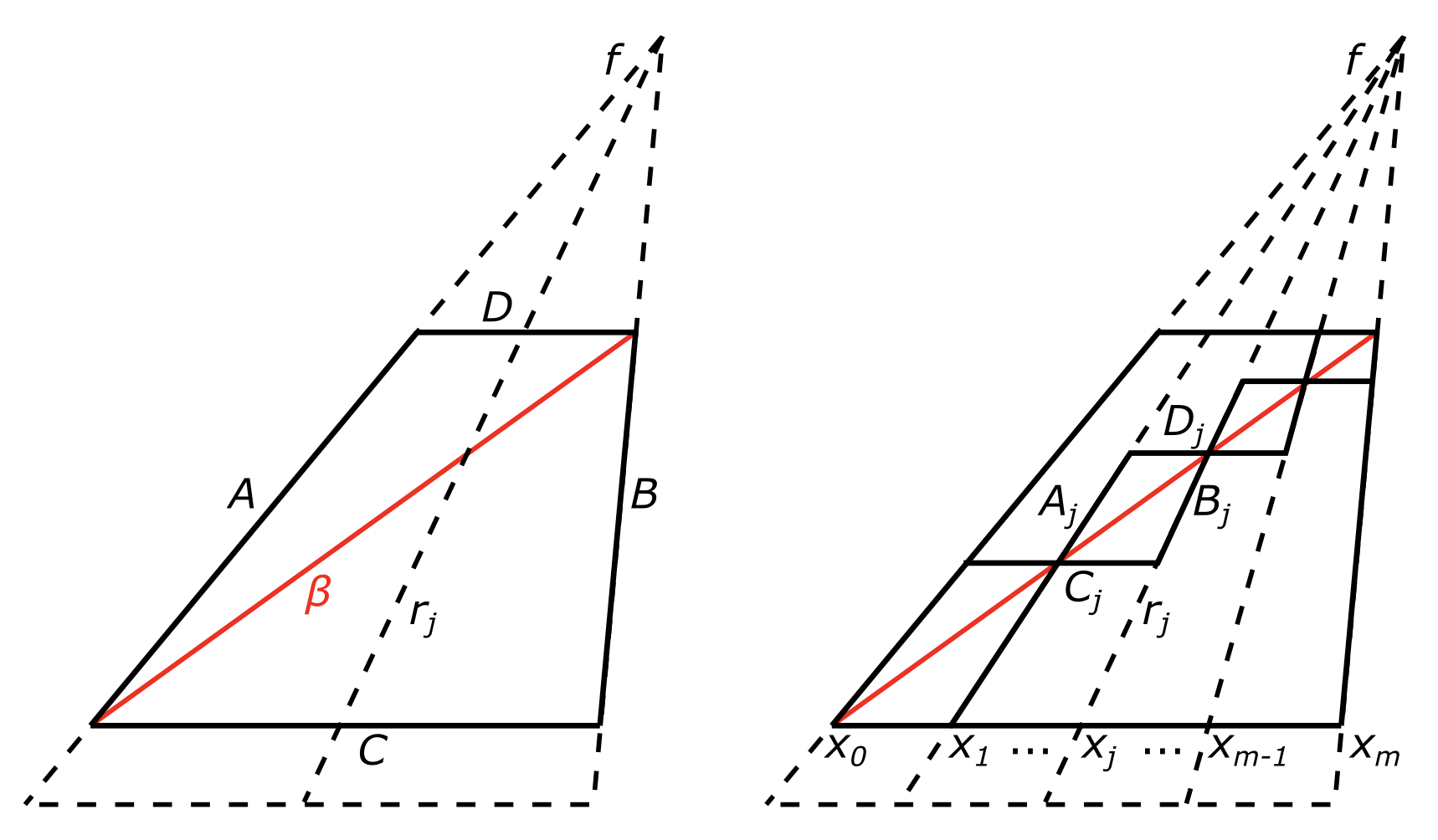}
	\caption[short]{The setting for Lemma \ref{lem_trapezoidbound}.} 
	\label{fig:trapezoid}
\end{figure}

We now proceed with the proof of Proposition \ref{prop:length_bound_cap}, starting with some notation.\\

If $w=\gamma(\phi_w,\theta_w)$, then due to symmetry we can restrict ourselves to $\phi_w,\theta_w\in(0,\pi/2]$ and $t\in[0,1)$.
	For if $\sin(\theta_w)=0$, then $w$ is one of the poles $p_n$ or $p_s$ and $\beta$ is contained in a straight line-segment with length bounded by 2.
	
	In the following we refer to the `northern separation' as the boundary between the north pole cap and the equatorial belt, i.e. $\s \cap \{z = 2/3\}$, and the `northern region' as the north pole cap. The `southern separation' and `southern region' are defined analogously.

Before we proceed by case distinction, we will introduce some points of interest:
$$p_{top}:=\gamma\big(\phi_w,\lvert \arccos(t)-\theta_w \rvert \big),\hspace{0.3cm} p_{bottom}:=\gamma\big(\phi_w,\arccos(t)+\theta_w\big)\in \partial C(w,t)$$ 
are closest to the north and south pole respectively.
Let $M_w\in\SO$ denote the rotation matrix, with axis of rotation given by $w$ and rotation angle of $\frac{\pi}{2}$. Set
$$p_{left}:=M_w(p_{top}),\hspace{0.5cm}
p_{right}:=M_w(p_{bottom}), $$ 
and note that both points have the same theta-angle $\theta_*$. Thus if we cut $C(w,t)$ by great circles going through $p_{top},p_{bottom}$ and $p_{left},p_{right}$, we obtain four parts of equal area.
 If $\partial C(w,t)$ crosses the northern separation, then it does so at 
 \begin{align*}
 	p_{ins1}&:=\gamma\Big(\phi_w-\arccos\Big(\frac{3t-2c_w}{\sqrt{5}s_w}\Big),\arccos\big(\frac{2}{3}\big)\Big),\\
 	p_{ins2}&:=\gamma\Big(\phi_w+\arccos\Big(\frac{3t-2c_w}{\sqrt{5}s_w}\Big),\arccos\big(\frac{2}{3}\big)\Big).
 \end{align*}
(These points were obtained using parametrizations \eqref{parametrization1} and \eqref{parametrization2}.) This implies that moving from $p_{top}$ to $p_{bottom}$ we leave the northern part at $p_{ins1}$.
We will use the image under $\Gamma$ of the points above (see Figure \ref{fig:sphereandgamma}), and we will denote them
$$p^\Gamma_j=\Gamma(p_j)\hspace{1cm}\mbox{for}\hspace{1cm}j\in\big\{top,bottom,left,right,ins1,ins2\big\},$$
and denote by $\beta_{1,2}$ the curves which are the image of the parametrizations \eqref{parametrization1}, \eqref{parametrization2} under $\Gamma$ -- note that it is not necessary to differentiate between $\beta_1$ and $\beta_2$; 
even though they will behave different in general, the bounds we give are valid for both. 
Thus in the sequel it is enough to consider the curve $\beta_1$ and double bounds in the end, when we talk about $\beta = \Gamma(\partial C)$.

We will next list all cases we have to consider in order to complete the proof and proceed thereafter by case distinction.
Cases that follow by symmetry (i.e. in the southern region) will not be considered.

\vspace{0.3cm}

I. Assume the cap satisfies $C(w,t)\subset C(\gamma(\phi,\frac{\pi}{2}),0)$ for some $\phi$.

\hspace{0.2cm}Ia) The cap is contained in the equatorial belt.

\hspace{0.2cm}Ib) The cap is contained in the northern region.

\hspace{0.2cm}Ic) The cap intersects the equatorial belt and the northern region, but 

\hspace{0.7cm} $p^\Gamma_{left},p^\Gamma_{right}$  are contained in the equatorial belt.

\hspace{0.2cm}Id) The cap intersects the equatorial belt and the northern region, where 

\hspace{0.6cm} $p^\Gamma_{left},p^\Gamma_{right}$ are contained in the northern region.

\hspace{0.2cm}Ie) The cap intersects all regions.

\vspace{0.3cm}

II. Assume the cap $C(w,t)$ contains only one pole, say the north pole.

\hspace{0.2cm}IIa) The cap is contained in the equatorial belt.

\hspace{0.2cm}IIb) The cap is contained in the northern region.

\hspace{0.2cm}IIc) The cap intersects the equatorial belt and the northern region, but 

\hspace{0.7cm} $p^\Gamma_{left},p^\Gamma_{right}$  are contained in the equatorial belt.

\hspace{0.2cm}IId) The cap intersects the equatorial belt and the northern region, where 

\hspace{0.6cm} $p^\Gamma_{left},p^\Gamma_{right}$ are contained in the northern region.

\hspace{0.2cm}IIe) The cap intersects all regions.

\medskip
\noindent
Before we start we want to state some additional properties of $\beta$ that will be useful.
By definition of the projection $\Gamma$, the curve $\beta $ is contained in the rectangle $R = [0,2] \times [-1/2,1/2]$. 
The equator corresponds a vertical line at height $0$.
The southern and northern separation lines are given by vertical lines at height $-1/4$ and $1/4$ respectively. 

Given $p_j^\Gamma$ and $p_k^\Gamma$, where $j,k \in \big\{top,bottom,left,right,ins1,ins2\big\}$, 
we use $\vert \langle p_j^\Gamma - p_k^\Gamma, e_1 \rangle \vert$ to denote the difference in the $x$-axis of the line \textit{inside} the cap (after applying $\Gamma$).
Note that due to the periodicity this is a slight abuse of notation. \\

\begin{description}
	\item[\textbf{Case Ia)}]  $\beta$ is a function of the $y$-coordinate and  monotone between the points $(p^\Gamma_{top},p^\Gamma_{left})$, $(p^\Gamma_{left}, p^\Gamma_{bottom})$, $(p^\Gamma_{bottom},p^\Gamma_{right})$, and $(p^\Gamma_{right},p^\Gamma_{top})$. 
	The difference of the height of any two of these points is bounded by $1/4$, while the width between $p_{right}^\Gamma$ and $p_{left}^\Gamma$ is bounded by $2$.
	Monotonicity and the above bounds for the height and width yield a length bound for any of the above `sectors' of $1+1/4$, so $\mbox{length}(\beta) \leq 5$. 
	
	\item[\textbf{Case Ib)}] Note that the images of boundaries of hemispheres going through the north and south poles are mapped to polygons; and in the northern region, the situation is like the setup in Lemma \ref{lem_trapezoidbound} with $f=\Gamma(p_n)$. Since hemispheres go through $\partial C(w,t)$ at most twice, we just have to cut $C(w,t)$ into four parts of equal area as mentioned further above in this proof, and consider the parts of $\beta$ belonging to these four pieces, and apply Lemma \ref{lem_trapezoidbound} four times. Since only the height and width of the image of $\beta$ matter for the bound of the lemma, we use the maximal slope 1 and set 
	$$a=2\sqrt{2}\ \langle p^\Gamma_{top}-p^\Gamma_{bottom},e_2 \rangle\leq \frac{1}{\sqrt{2}},
	\hspace{0.3cm}$$ 
	and 
	$$c=2\lvert\langle p^\Gamma_{left}-p^\Gamma_{right},e_1 \rangle\rvert\leq 4.$$
 Thus
	$$\mbox{length}(\beta)\leq a+c = 4+\frac{1}{\sqrt{2}}.$$
	
	\begin{figure}[h!]
\begin{center}
\includegraphics[scale=0.5]{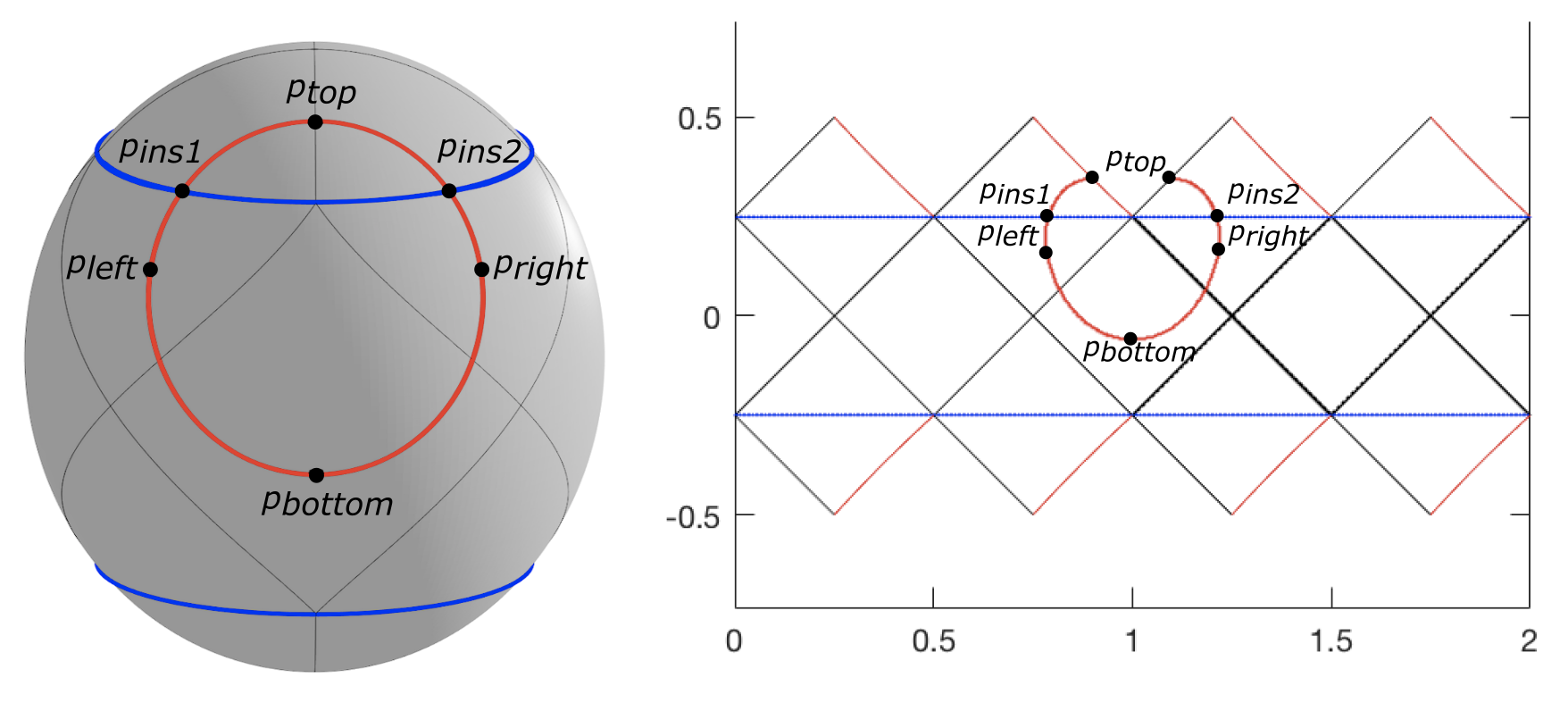}
\caption{The setting for case Ic).}
\label{fig:sphereandgamma}
\end{center}
\end{figure}

	\item[\textbf{Case Ic)}] In the northern region we proceed as in the case Ib), where we use 	\begin{align*}
		a_1=2\sqrt{2}\ \langle p^\Gamma_{top}-p_{ins1}^\Gamma,e_2 \rangle
		\qquad \text{and} \qquad
		 c_1=\lvert \langle p^\Gamma_{ins1}-p^\Gamma_{ins2},e_1 \rangle \rvert.
	\end{align*}
Two applications of Lemma \ref{lem_trapezoidbound} yield a bound of $a_1 + c_1$ for the northern part. 
(Note that $\langle p^\Gamma_{top}-p_{ins1}^\Gamma,e_2 \rangle = \langle p^\Gamma_{top}-p_{ins2}^\Gamma,e_2 \rangle$, since both $p^\Gamma_{ins1}$ and $p^\Gamma_{ins2}$ have the same height.)

The part of $\beta$ falling into the equatorial belt is a function of the $y$-coordinate, first monotonously decreasing from $p^\Gamma_{ins1}$ to $p^\Gamma_{left}$, and then monotonously increasing from $p^\Gamma_{left}$ to $p^\Gamma_{bottom}$.
So we can bound the length by 
\[
\langle p^\Gamma_{ins1} - p^\Gamma_{left}, e_2 \rangle
+
\vert \langle p^\Gamma_{ins1} - p^\Gamma_{left}, e_1 \rangle \vert 
= a_2 + c_2,
\]
and 
\[
\langle p^\Gamma_{left} - p^\Gamma_{bottom}, e_2 \rangle
+
\vert \langle p^\Gamma_{left} - p^\Gamma_{bottom}, e_1 \rangle \vert 
= a_3 + c_3.
\]
In a similar way we obtain $a_4,c_4$ and $a_5,c_5$ for the parts of $\beta$ which are `between' $p^\Gamma_{bottom}, p^\Gamma_{right}$ and $p^\Gamma_{right}, p^\Gamma_{ins2}$ respectively.  

To obtain the final bound we estimate $\sum_{i=1}^{5}(a_i+c_i)$.
Thereto we observe that $\sum_{i=1}^{5}c_i = 2\vert \langle p^\Gamma_{left}-p^\Gamma_{right},e_1 \rangle \vert  \leq 4$. 

Moreover, $a_1 \leq 1/\sqrt{2}$ and  
$\sum_{i=2}^{5}a_i = 2\langle p^\Gamma_{ins1} - p^\Gamma_{bottom}, e_2 \rangle  \leq 1.$
This gives a total bound of 
\[
\text{length}(\beta) \leq 5+ \frac{1}{\sqrt{2}} .
\]

	\item[\textbf{Case Id)}] This case is very similar to case Ic).
	To bound the `upper half' of $\beta$, i.e., the part above the line through $p^\Gamma_{left}$ and $p^\Gamma_{right}$ we obtain by two applications of Lemma \ref{lem_trapezoidbound} 
	\[
	2\sqrt{2}\langle p^\Gamma_{top}-p^\Gamma_{left}, e_2 \rangle +
	\vert \langle  p^\Gamma_{left} - p^\Gamma_{right}, e_1 \rangle \vert =  a_1 + c_1.  
	\]
	Moreover, the part between $p^\Gamma_{left}$ and $p^\Gamma_{ins1}$ is bounded by Lemma \ref{lem_trapezoidbound} by 
	\[
	\sqrt{2}\langle p^\Gamma_{left}-p^\Gamma_{ins1}, e_2 \rangle +
	\vert \langle  p^\Gamma_{left} - p^\Gamma_{ins1}, e_1  \rangle \vert =  a_2 + c_2.
	\]
	Similarly we define $a_3, c_3$ for the part between $p^\Gamma_{right}$ and $p^\Gamma_{ins2}$.
	
	The equatorial belt region is treated as before, yielding a bound of 
	\[
	2\langle p^\Gamma_{ins1} - p^\Gamma_{bottom} ,e_2\rangle + \vert \langle p^\Gamma_{ins1} - p^\Gamma_{ins2}, e_1 \rangle \vert = a_4 + c_4.  
	\] 
	Note that $\sum_{i=1}^4 c_i = 2 \vert \langle p^\Gamma_{left} - p^\Gamma_{right}, e_1 \rangle \vert\leq 4$.
	
	Further, $\sum_{i=1}^3a_i = 2\sqrt{2}\langle p^\Gamma_{top} - p^\Gamma_{ins1}, e_2 \rangle \leq 1/\sqrt{2}$
	 and $a_4 \leq 1$. In the end we sum up everything and get the same bound as in case Ic).
	\item[\textbf{Case Ie)}] 
	This case can be treated using the same arguments as for cases Ic) and Id). 
	The bound for the width, i.e., the sum of $c_i$'s, is still bounded by $4$. 
	However, since the cap also goes through the southern region the bound for the height increases by $1/\sqrt{2}$, compared to the previous cases. 
	In conclusion we end up with $\mbox{length}(\beta) \leq 5+ \sqrt{2} $.\\
	
	\item[\textbf{Cases II*)}] In the cases II*) the cap contains the north pole. 
It suffices to treat caps where the pole is part of the interior of the cap, as the cases I*) cover the cases where the pole is on the boundary. 
	
	\item[\textbf{Case IIa)}]The curve $\beta_1$ is a function of the $y$-coordinate and is monotone between the points $(p^\Gamma_{top},p^\Gamma_{bottom})$  - note that the co-domain of $\Gamma$ is a cylinder. Thus the  $\mbox{length}(\beta_1)$ in this case is bounded by $3/2$.
	So the total bound in this case is $3$.
	
	\item[\textbf{Case IIb)}] 
	We note that $\beta_1$ is monotonic between $p^\Gamma_{top}$ and $p^\Gamma_{bottom}$.
	Hence Lemma \ref{lem_trapezoidbound} is applicable, and using the maximal slope $1$ it follows 
	\[
	\mbox{length}(\beta_1) \leq 2\sqrt{2} \langle p^\Gamma_{top} - p^\Gamma_{bottom},e_2 \rangle 	+ 2  
	\leq  
	\frac{1}{\sqrt{2}} +2;	\]
	\item[\textbf{Case IIc)}] 
	We use similar arguments as in cases Ic) and IIb).
	The total width is bounded by $4$, whereas the height bound is $1/\sqrt{2}+1$.
	Thus $\mbox{length}(\beta) \leq 5+ 1/\sqrt{2}$.
	
\item[\textbf{Case IId)}] 
Similar as before we have a width bound of $4$ and a height bound of $1+1/\sqrt{2}$, yielding $\mbox{length}(\beta) \leq 5+ 1/\sqrt{2}$.

\item[\textbf{Case IIe)}] 
We use the same arguments as in cases Ie) and IIb) to obtain a bound of $\mbox{length}(\beta) \leq 5+\sqrt{2}$. 
\end{description}

This ends the case distinction and hence the proof of Proposition \ref{prop:length_bound_cap}. \qed

%
%
%

\subsection{Curvature calculations}

\label{subsec:curvcalc}

For the proof of Corollary \ref{cor_capsconvex} we will now calculate the signed curvature $\kappa$ of the planar curve $\Gamma(\partial C) \subset R$ and give a universal bound on the number of zeros. Here $x$ and $y$ are the coordinates of $\Gamma \circ \gamma \circ (\text{param.}  \eqref{parametrization1})$ and $x'$ and $y'$ denote the derivatives of $x$ and $y$ with respect to $\theta$:
 $$\kappa=\frac{x'y''-y'x''}{\big((x')^2+(y')^2\big)^{3/2}}.$$

We see immediately that if $w$ is a pole, then $\kappa(\theta)\equiv0$ since the $y$ value is constant. Similarly, if $\partial C$ goes through both poles, then $\kappa(\theta)\equiv0$ since the $x$ value is constant. These caps are treated separately in Corollary \ref{cor_capsconvex}, and hence we assume in the subsequent calculations that $\theta_w\in (0,\frac{\pi}{2}]$, and $t>0$ in case $\theta_w=\frac{\pi}{2}$.

\subsubsection{Equatorial belt region}

\label{eqbelt}

 From the definition of $\Gamma$ it follows that the part of $\Gamma(\partial C)$ that belongs to the equatorial belt has the following parametrization (note that $\sqrt{5}/3\leq\sin(\theta)\leq 1$):
$$x(\theta)=\frac{1}{\pi}\arccos\Big(\frac{t-\cos(\theta)c_w}{\sin(\theta)s_w}\Big)+\frac{\phi_w}{\pi},$$
and
$$y(\theta)=\frac{3}{8}\cos(\theta).$$
The derivatives needed for the curvature are as follows
\begin{align*}
x'&=\frac{1}{\pi}\frac{-\sin(\theta)s_w}{\sqrt{\sin(\theta)^2s_w^2-(t-\cos(\theta)c_w)^2}}\frac{\sin(\theta)^2s_wc_w-(t-\cos(\theta)c_w)\cos(\theta)s_w}{\sin(\theta)^2s_w^2}\\
&=-\frac{1}{\pi A}\frac{c_w-t\cos(\theta)}{\sin(\theta)},
\end{align*}
where we use the abbreviation $$A=\sqrt{\sin(\theta)^2s_w^2-(t-\cos(\theta)c_w)^2}.$$
Further
\begin{align*}
\pi x''&=\frac{\sin(\theta)\cos(\theta)s_w^2-(t-\cos(\theta)c_w)\sin(\theta)c_w}{\big(\sin(\theta)^2s_w^2-(t-\cos(\theta)c_w)^2\big)^{3/2}}\frac{c_w-t\cos(\theta)}{\sin(\theta)}\\
&\hspace{1cm}-\frac{t\sin(\theta)^2-(c_w-t\cos(\theta))\cos(\theta)}{\sqrt{\sin(\theta)^2s_w^2-(t-\cos(\theta)c_w)^2}\sin(\theta)^2}\\
&=\frac{\big(\cos(\theta)-tc_w\big)\big(c_w-t\cos(\theta)\big)}{A^3}-\frac{t-c_w\cos(\theta)}{A\sin(\theta)^2}
\end{align*}
and
\begin{align*}
y'(\theta)=-\frac{3}{8}\sin(\theta) \hspace{0.5cm} \text{and} \hspace{0.5cm} y''(\theta)=-\frac{3}{8}\cos(\theta).
\end{align*}
Thus $x'y''-y'x''=$
\begin{align*}
&\frac{3}{\pi 8 A}\Big(\big(c_w-t\cos(\theta)\big)\big(\cot(\theta)+\big(\cos(\theta)-tc_w\big)\frac{\sin(\theta)}{A^2}\big)-\frac{t-c_w\cos(\theta)}{\sin(\theta)}\Big)\\
&=\frac{3}{\pi 8 A}\Big(\cot(\theta)2c_w+\big(\cos(\theta)-tc_w\big)\big(c_w-t\cos(\theta)\big)\frac{\sin(\theta)}{A^2}-\frac{t(1+\cos(\theta)^2)}{\sin(\theta)}\Big),
\end{align*}
and 
$$
x'^2+y'^2=\frac{1}{\pi^2}\frac{1}{A^2}\frac{\big(c_w-t\cos(\theta)\big)^2}{\sin(\theta)^2}+\frac{9}{64}\sin(\theta)^2.
$$
Assuming $A \neq 0$, we obtain
$$ 
\kappa(\theta)=\frac{
	3A^2\Big(\cot(\theta)2c_w+\frac{1}{A^2}\big(\cos(\theta)-tc_w\big)\big(c_w-t\cos(\theta)\big)\sin(\theta)-\frac{t(1+\cos(\theta)^2)}{\sin(\theta)}\Big)
}{
	\pi 8 \Big(\frac{1}{\pi^2}\frac{\big(c_w-t\cos(\theta)\big)^2}{\sin(\theta)^2}+A^2\frac{9}{64}\sin(\theta)^2  \Big)^\frac{3}{2}.
}
$$
This equation can be continuously extended to $A = 0$. By the trigonometric angle sum formula,
\begin{align*}
A^2&=\big[\sin(\theta)s_w-t+\cos(\theta)c_w\big]\big[\sin(\theta)s_w+t-\cos(\theta)c_w\big]\\
&=\big[\cos(\theta-\theta_w)-t\big]\big[t-\cos(\theta+\theta_w)\big]
\end{align*}
is zero only for $\theta\in[\arccos(\tfrac{2}{3}),\arccos(-\tfrac{2}{3})]\cap\{\arccos(t)+\theta_w,\lvert \arccos(t)-\theta_w \rvert\}.$ We also make sure that $c_w-t\cos(\theta)$ is not zero when $A$ is, thus plugging in $\theta = \arccos(t)+\theta_w$ and $\theta = \lvert \arccos(t)-\theta_w \rvert$:
\begin{align*}
	c_w-t(c_wt-s_w\sqrt{1-t^2}) &= c_w(1-t^2)+s_wt\sqrt{1-t^2},\\
	c_w-t(c_wt+s_w\sqrt{1-t^2}) &= \sqrt{1-t^2}\sin\big(\arccos(t)-\theta_w\big).
\end{align*}
Since $0\leq t<1$, the first value above is always positive, while the second is not zero since $\theta=\arccos(t)-\theta_w\in\{0,\pi\}$ is not in the domain of the parametrization.

 Thus the question if the curvature changes sign is reduced to finding $\theta$ such that
$$ 
A^2\cos(\theta)2c_w+\sin(\theta)^2\big(\cos(\theta)-tc_w\big)\big(c_w-t\cos(\theta)\big)-A^2t(1+\cos(\theta)^2)=0
$$
which can be expanded in the notation $\sin(\theta)=s_\theta$, $\cos(\theta)=c_\theta$ as
\begin{align*}
&2c_\theta c_w s_\theta^2 s_w^2-2c_\theta c_w t^2+4t c_\theta^2 c_w^2-2 c_\theta^3 c_w^3 +s_\theta^2 c_\theta c_w-tc_\theta^2 s_\theta^2-tc_w^2s_\theta^2+t^2c_\theta c_ws_\theta^2\\
&-ts_\theta^2s_w^2+t^3-2t^2c_\theta c_w+tc_\theta^2c_w^2 - tc_\theta^2 s_\theta^2s_w^2+t^3c_\theta^2-2t^2c_\theta^3 c_w+tc_\theta^4c_w^2 = 0.
\end{align*}
Using that $s_\theta^2=1-c_\theta^2$, we see that
this is a fourth degree polynomial in $c_\theta$ and has at most 4 roots for $c_\theta\in[-\tfrac{2}{3},\tfrac{2}{3}]$. Therefore, there are at most $8$ zeros of curvature.
%


\subsubsection{North and south pole caps}
\label{NSpole}

Here we discuss the curvature of $\Gamma(\partial C)$ in the north and south pole caps. 
From the definition of $\Gamma$ it follows that the part of $\Gamma(\partial C)$ that falls in the north pole cap can be parametrized as follows. (The parametrization of the south pole cap is analogous, so we just focus on the north pole cap.)
$$x(\theta)=\frac{\phi}{\pi}-\frac{1}{\pi}\big(1-\sqrt{1-\cos(\theta)}\sqrt{3}\big)\cdot(\phi\hspace{-0.2cm}\mod \tfrac{\pi}{2}-\tfrac{\pi}{4}),$$
and by trigonometric half-angle identities
$$y(\theta)=\frac{1}{4}\big(2-\sqrt{1-\cos(\theta)}\sqrt{3}\big)=\frac{1}{2}-\sin\Big(\frac{\theta}{2}\Big)\frac{\sqrt{6}}{4}.$$
It follows that 
\begin{align*}
y'(\theta)=-\cos\Big(\frac{\theta}{2}\Big)\frac{\sqrt{6}}{8} \hspace{0.5cm} \text{and} \hspace{0.5cm} y''(\theta)=\sin\Big(\frac{\theta}{2}\Big)\frac{\sqrt{6}}{16}.
\end{align*}
As for the derivative of $x$, note that the factor $\phi\hspace{-0.2cm}\mod \frac{\pi}{2} $ can be written as $\phi(\theta)-j\frac{\pi}{2}$ for appropriate $j$ and $\theta$, thus in  the derivative we can ignore that part:
\begin{align*}
x'&=\frac{\phi'}{\pi}\Big(1 -1+\sqrt{1-\cos(\theta)}\sqrt{3}\Big)-\frac{4y'}{\pi}\big(\phi\hspace{-0.2cm}\mod \tfrac{\pi}{2}-\tfrac{\pi}{4}\big)\\
&=16\frac{\phi'y''}{\pi}-\frac{4y'}{\pi}\big(\phi\hspace{-0.2cm}\mod \tfrac{\pi}{2}-\tfrac{\pi}{4}\big)
\end{align*}
and thus, making use of the fact that $y''' = \frac{1}{4} y'$,
\begin{align*}
x''&=16\frac{\phi''y''-\phi'y'/4}{\pi}-\frac{4y''}{\pi}\big(\phi\hspace{-0.2cm}\mod \tfrac{\pi}{2}-\tfrac{\pi}{4}\big)-\frac{4y'\phi'}{\pi}\\
&=\frac{16}{\pi}\phi''y''-\frac{4y''}{\pi}\big(\phi\hspace{-0.2cm}\mod \tfrac{\pi}{2}-\tfrac{\pi}{4}\big)-\frac{8}{\pi}y'\phi'.
\end{align*}
Note that the term 
$$
(y')^2=\frac{3}{8^2}(1+\cos(\theta))
$$
is zero only if $\theta=\pi$, but since we restricted ourselves to the case $0\leq\theta_w\leq\frac{\pi}{2}$ and $0\leq t<1$ due to symmetry, this would mean that $t=0$ and $\theta_w = \frac{\pi}{2}$, but these caps can be disregarded here because they are dealt with separately in   Corollary \ref{cor_capsconvex}. Thus
\begin{equation}\label{eq_curvatureNPregion}
	\kappa(\theta)=\frac{16\frac{\phi'(y'')^2}{\pi}-
		\frac{16}{\pi}\phi''y''y'+\frac{8}{\pi}(y')^2\phi'}{\Big[\Big(16\frac{\phi'y''}{\pi}-\frac{4y'}{\pi}\big(\phi\hspace{-0.2cm}\mod \tfrac{\pi}{2}-\tfrac{\pi}{4}\big)\Big)^2+\frac{3}{8^2}(1+\cos(\theta))\Big]^{3/2}}.
\end{equation}

To show that the curvature is well defined, we proceed as in the previous section - with one exception: if the north pole is in $\partial C$, then $\theta=0$ is in the domain of parametrization. If this case  occurs, then we can use the parametrization \eqref{eq_northpoleparametrization}
and obtain subsequently with $\tau= \frac{ t}{\sqrt{1-t^2}}$ (note that $t = 0$ is excluded, because $\theta_w = \frac{\pi}{2}$ in that case):
$$\phi'=-\frac{1}{2}\frac{\tau+\tau\tan\big(\tfrac{\theta}{2}\big)^2}{\sqrt{1-\tau^2\tan\big(\tfrac{\theta}{2}\big)^2}} $$
and using the product formula for the derivative,
\begin{align*}
	\phi''&=-\frac{\tau}{4}\frac{\tau^2\tan\big(\tfrac{\theta}{2}\big)\big(1+\tan\big(\tfrac{\theta}{2}\big)^2\big)}{\Big(1-\tau^2\tan\big(\tfrac{\theta}{2}\big)^2\Big)^{3/2}} \big(1+\tan\big(\tfrac{\theta}{2}\big)^2\big)-\frac{\tau}{2}\frac{\tan\big(\tfrac{\theta}{2}\big)\big(1+\tan\big(\tfrac{\theta}{2}\big)^2\big)}{\sqrt{1-\tau^2\tan\big(\tfrac{\theta}{2}\big)^2}}\\
	&=-\frac{\tau}{2}\frac{\tan\big(\tfrac{\theta}{2}\big)\big(1+\tan\big(\tfrac{\theta}{2}\big)^2\big)}{\sqrt{1-\tau^2\tan\big(\tfrac{\theta}{2}\big)^2}}\left(\frac{\tau^2}{2}\frac{\big(1+\tan\big(\tfrac{\theta}{2}\big)^2\big)}{1-\tau^2\tan\big(\tfrac{\theta}{2}\big)^2} +1 \right).
\end{align*}
Thus the curvature is hence well defined for $0<\theta<2\arctan\big(\frac{\sqrt{1-t^2}}{t}\big)$ and can be continuously extended to zero by
$$\kappa(0)=\frac{16\frac{\phi'\cdot 0}{\pi}-
	\frac{16}{\pi}\cdot 0\cdot y''y'-\frac{8}{\pi}\frac{6}{8^2}\frac{\tau}{2}}{\Big[8^2\tau^2\frac{\phi'\cdot 0}{\pi}+\frac{6}{4\pi^2}\big(\phi_w\hspace{-0.2cm}\mod \tfrac{\pi}{2}-\tfrac{\pi}{4}\big)^2+\frac{6}{8^2}\Big]^{3/2}}\leq
-\frac{2^3}{\pi\sqrt{3}}\frac{t}{\sqrt{1-t^2}}.$$
A similar calculation, using the parametrization \eqref{eq_southpoleparametrization}, yields that $\theta=\pi$ can be treated in the same fashion. We hence also know that  the curvature is not zero for $\theta\in\{0,\pi\}$, and by continuity, also not in an neighborhood of these values. 

The curvature $\kappa(\theta)$ thus changes sign, if at all, at
$$\frac{16\phi'}{\pi}(y'')^2-
\frac{16}{\pi}\phi''y''y'+\frac{8\phi'}{\pi}(y')^2=0.
$$
We note that 
$$
y''y'=-\sin(\theta)\frac{3}{128}\hspace{0.5cm}\mbox{and}\hspace{0.5cm}(y'')^2=\frac{3}{16^2}(1-\cos(\theta)),
$$
thus we equivalently want, letting $\xi = \frac{\phi}{\pi}$ and multiplying by $\sin(\theta)A^3$,
$$
\sin(\theta)A^3\left(\xi'\frac{3}{16}(1-\cos(\theta))+
\frac{3\sin(\theta)}{8}\xi''+\frac{3}{8}(1+\cos(\theta))\xi'\right)=0;
$$
or if we multiply by $\frac{16}{3}$
\begin{equation*}
\sin(\theta)A^3\Big(2\sin(\theta)\xi''+(3+\cos(\theta))\xi'\Big)=0.
\end{equation*}
Next we use the expansion of $\xi' = x'(\theta)$ (this refers to the $x$ in \textsection \ref{eqbelt}) to obtain a polynomial in $c_\theta$ of degree at most five
which has at most 5 roots for $c_\theta\in[\tfrac{2}{3},1]$.

\subsubsection{Curvature for the other parametrization}
\label{otherparametrization}
If we had chosen to work with parametrization \eqref{parametrization2} instead, then in the equatorial belt region only the sign of $\kappa$ would have changed, as the derivatives of the $x$-values flip signs. In the polar cap regions the curvatures of the two parametrizations have opposing signs as can be seen from formula \eqref{eq_curvatureNPregion}, and are  zero for the same values $\theta_0$. In either region, a zero for the curvature of one paramatrization is also a zero for the other. Since the $y$-values are the same for both parametrizations, such zeros have the same $y$-value. \\

Note that the value of 36 in Corollary \ref{cor_capsconvex} is obtained by summing the total number of zeros of curvature for one parametrization, and then doubling it.

\section{Appendix: Formulas for Boundaries and Centers of Pixels}

Here we detail (for the reader's convenience) an indexing of pixels in the HEALPix tessellation at an arbitrary level $\ell$, and give formulas for their centers, in other words the HEALPix point set. This is an alternative to the indexing presented in \cite{Gorski} and may be useful for programming. 

As before, we use $\theta \in [0, \pi)$ to denote angle with the north pole, so $z = \cos(\theta)$. $\phi \in [0,2\pi)$ denotes the angle with the $x$ axis on the $xy$-plane. We define a pixel by $(s,k,\ell,r,c)$. Each index is described below.

\begin{itemize}
\item $s \in \{ -1, 0, 1 \}$ denotes whether the pixel is one of the north base pixels, one of the equatorial base pixels, or one of the south base pixels ($s =1$, $s=0$, $s=-1$, respectively).
\item $k \in \{0,1,2,3\}$ defines which of the four base pixels on level $s$ the pixel is in. For $s= \pm 1$, $k =0$ corresponds to the base pixel whose boundary touches $z = \pm \frac{2}{3}$ at the points $\phi = 0$ and $\phi = \frac{\pi}{2}$. For $s= 0$, $k = 0$ corresponds to the base equatorial pixel centered on the arc $\phi = 0$. For $k > 0$, this value of $k$ indicates a rotation by $\frac{k \pi}{2}$ radians in $\phi$.
\item $\ell \in \mathbb{N}_0$ denotes the resolution of the pixelation. In each base pixel, there are $4^{\ell}$ pixels of resolution $\ell$, in a grid with $2^{\ell}$ rows and $2^{\ell}$ columns.
\item $r \in \{1, \cdots, 2^{\ell}\}$ indexes the $2^{\ell}$ parallels to the northeast boundary in the grid of subpixels of the base pixel. In the north pole cap and equatorial belt, the indexing starts on the northeastern side of the base pixel and ends at the southwestern side; $r=1$ denotes the pixels that have the northeastern boundary of the base pixel as their boundary. In the south pole cap, the indexing starts on the southeastern side of the base pixel and ends at the northwestern side.
\item $c \in \{1, \cdots, 2^{\ell}\}$ indexes the $2^{\ell}$ parallels to the northwest boundary in the grid of subpixels of the base pixel.  In the north pole cap and equatorial belt, the indexing starts on the northwestern side of the base pixel; $c =1$ denotes the pixels that have the northwestern boundary of the base pixel as their boundary. In the south pole cap, the indexing starts on the southwestern side of the base pixel and ends at the northeastern side.
\end{itemize}

\emph{Example.} In Figure \ref{fig:pixelboundaries} below, assume that $\phi = 0$ is the arc on which the dark gray base pixel in the picture is centered, and that the north pole is the top vertex of the light gray base pixel. Then the pixel shaded in red has index $(1,0,3,5,3)$.

\begin{figure}[h!]
\begin{center}
\includegraphics[scale=0.5]{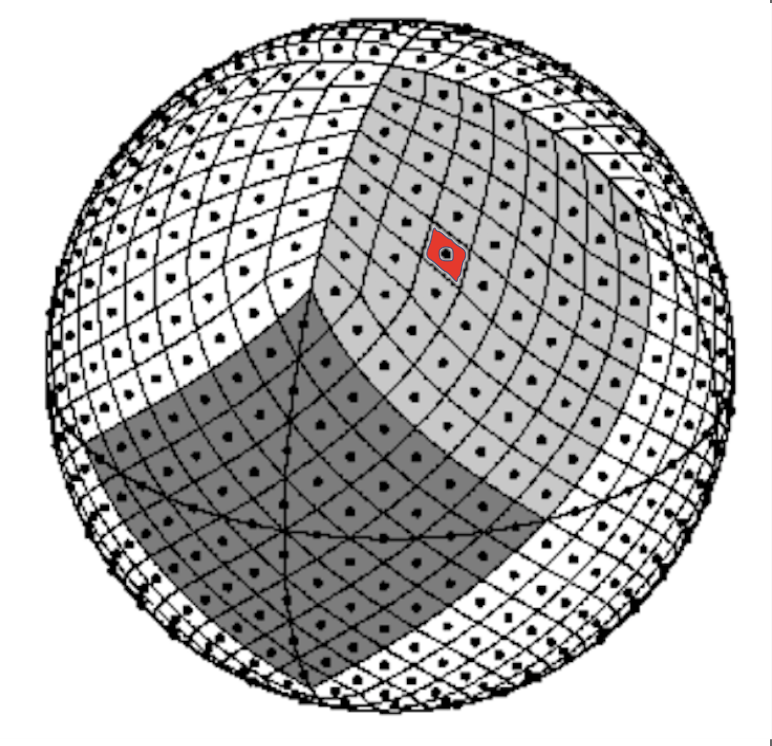}
\caption{ \copyright AAS. Reproduced (with slight modification) with kind permission from the authors of \cite{Gorski} and the Astrophysical Journal. See DOI: 10.1086/427976.}
\label{fig:pixelboundaries}
\end{center}
\end{figure}

Write $L = 2^{\ell}$.
The pixel boundaries of resolution parameter $\ell$ are obtained via curves $m^e_{j,\ell}$ from north to south, and curves $p^e_{j,\ell}$ from south to north. Here we use a slightly different indexing system from that used in \textsection \ref{subsec:boundaries}.  Let $\phi_k = \phi-\frac{k\pi}{2}$. Then 
\[m_{j,k,\ell}^e \sim \cos(\theta) = \frac{2}{3}-\frac{4j}{3L}-\frac{8\phi_k}{3\pi}\hspace{0.5cm} \text{  and  } \hspace{0.5cm} p_{j,k,\ell}^e \sim \cos(\theta) =  \frac{2}{3}-\frac{4j}{3L}+\frac{8\phi_k}{3\pi}.\]

A pixel's center $(z, \phi)$ is given by the shared vertex of the four subpixels of pixel $(s, k, \ell, r, c)$ in the next-finest tessellation at level $\ell+1$. We give explicit formulas below.

\subsection{Pixel centers contained in the equatorial belt}

\subsubsection{Pixels with s = 0}

%

Let $s = 0$ in the pixel indexing, so that the base pixel it is contained in lies entirely in the equatorial belt. The center of pixel $(0, k, \ell, r, c)$ , occurs at intersection point of the lines $m_{(2r-1),k,\ell+1}$ and $p_{(2c-1),k,\ell+1}$, which is at
$$\phi=\frac{\pi(r-c)}{4L}+\frac{k\pi}{2}$$
and
$$z =\frac{2}{3} + \frac{2(r+c-1)}{3L}.$$

\subsubsection{Pixels with s = 1}
We now describe centers of pixels of the form $(1,k,\ell,r,c)$ which are contained entirely in the equatorial belt, which occurs when $r+c > L+1$. (For the $s=-1$ case, the pixel centers are obtained by simply changing the sign of the $z$-value.)
%

The pixel's center $(z, \phi)$ is given by the shared vertex of the four subpixels of pixel $(1, k, \ell, r, c)$ in the next-finest tessellation at level $\ell+1$. This occurs at intersection point of the lines $m_{(2r-1),(k+1)_4,\ell+1}$ and $p_{(2c-1),k,\ell+1}$, which is at
$$\phi = \frac{\pi(r-c)}{4L}+\frac{(2k+1)\pi}{4}$$
and
$$z = \frac{4}{3}-\frac{2(r+c-1)}{3L}.$$



\subsection{Pixels contained in north and south caps or on their boundaries}


We concern ourselves only with the pixel centers that lie in the north pole cap or on its boundary, which occurs when $r+c \le L+1$ (again, the centers for pixels in the south pole cap are obtained by changing the sign of the $z$-value).

The center of pixel $(1,k,\ell,r,c)$ is given by the intersection of the curves $p_{2c-1,k,\ell}^n$ and $m_{2r-1,k,\ell}^n$, which occurs at 
$$\phi = k\frac{\pi}{2} + \frac{\pi(2c-1)}{4(r+c-1)}$$
and
$$z=1- \frac{(r+c-1)^2}{3L^2}.$$

\section*{Acknowledgements}
The first author is supported by long term structural funding -- Methusalem grant of the Flemish Government; the Austrian Science Fund (FWF): F5503 ``Quasi-Monte Carlo Methods''; and FWF: W1230 ``Doctoral School Discrete Mathematics.'' The second author is supported by Deutsche Forschungsgemeinschaft (DFG) within project 432680300 - SFB 1456 (subproject B02) and the Johannes Kepler University Linz. The third author is supported by DMS-2054606 through the US National Science Foundation. The authors thank Ryan Matzke for suggesting we collaborate and for many discussions on the subjects presented in this paper, including substantial work on the material in the appendix. We also thank Dmitriy Bilyk for helpful comments and suggestions. We thank  K.M. Górski, E. Hivon, A.J. Banday, B.D. Wandelt, F.K. Hansen, M. Reinecke and M. Bartelmann for the permission to use their graphics, as well as AAS publishing and The Astrophysical Journal. 

\end{document}